\documentclass[12pt]{amsart}
\usepackage{amssymb,amscd,amsmath}
\usepackage{fullpage}
\usepackage{tikz-cd}
\usepackage[disable]{todonotes}
\usepackage{verbatim} %For long comments
\usetikzlibrary{intersections}
\usepackage{ogonek}

\usepackage{amsthm}
\theoremstyle{definition}
\newtheorem{thm}{Theorem}[section]
\newtheorem{lem}[thm]{Lemma}
\newtheorem{prop}[thm]{Proposition}
\newtheorem{cor}[thm]{Corollary}

\theoremstyle{definition}
\newtheorem{dfn}{Definition}[section]
\newtheorem{exmp}{Example}[section]

\theoremstyle{remark}
\newtheorem*{rem}{Remark}

	\newcommand{\R}{\mathbb{R}}
	\newcommand{\N}{\mathbb{N}}
	\newcommand{\C}{\mathbb{C}}
	\newcommand{\Z}{\mathbb{Z}}
	\newcommand{\Q}{\mathbb{Q}}
	\newcommand{\h}{\mathcal{H}}

	\newcommand{\abs}[1]{\left|#1\right|}

	\newcommand{\dv}{\operatorname{div}}	
		
	\newcommand{\Eis}{\text{Eis}}	
	\newcommand{\Res}{\text{Res}}	
	\newcommand{\GL}{\mathrm{GL}}
	\newcommand{\SL}{\mathrm{SL}}
	\DeclareMathOperator{\ord}{ord}
    \renewcommand{\Re}{\text{Re}\,}
    
    \renewcommand{\mod}{\text{ mod}\,}
	
	\newcommand{\sabcd}[4]{\left(\begin{smallmatrix}#1&#2\\#3& #4\end{smallmatrix}\right)}

	\title{Mahler measures of elliptic modular surfaces}
	\author{Fran\c{c}ois Brunault and Michael Neururer}
	\usepackage{hyperref}
	\hypersetup{pdfauthor={Fran\c{c}ois Brunault and Michael Neururer},%
	            pdftitle={Elliptic Mahler measures},%
	            pdfsubject={Number Theory},%
	            pdfcreator={XeLaTeX},%
	            colorlinks=true,%
	            linkcolor=[rgb]{0,.6,1},
	            urlcolor = blue,
	            citecolor=blue}

\begin{document}
%\maketitle
%\tableofcontents



\begin{abstract}
In this article we develop a new method for relating Mahler measures of three-variable polynomials that define elliptic modular surfaces to $L$-values of modular forms. Using an idea of Deninger, we express the Mahler measure as a Deligne period of the surface and then apply the first author's extension of the Rogers--Zudilin method for Kuga--Sato varieties, to arrive at an $L$-value.
\end{abstract}

\maketitle

\section{Introduction}

Let $P\in\C(X_1,\ldots,X_n)$ be a nonzero Laurent polynomial. The (logarithmic) Mahler measure $m(P)$ of $P$ is defined as the average of $\log|P|$ on the $n$-torus $T^n = (S^1)^n$:
	\begin{align*}
	m(P)=\frac{1}{(2\pi i)^n}\int_{T^n} \log|P(z_1,\ldots,z_n)|\frac{dz_1}{z_1} \wedge \ldots \wedge \frac{dz_n}{z_n}.
	\end{align*}

If $n=2$ and $P$ defines an elliptic curve $E$, the work of Boyd \cite{Boyd1998} and Deninger \cite{Deninger1997} shows that under certain additional conditions on $P$, the Mahler measure $m(P)$ should be related to the $L$-value of $E$ at $s=2$. Boyd conjectured many remarkable relations of the form
\[
m(P) = r\Lambda(E,2)
\]
where $r\in\Q^\times$ and $\Lambda(E,s)$ denotes the completed $L$-function of $E$. While Boyd's conjectures remain open in general, a number of special cases have been established in recent years, see e.g. \cite{Rodriguez-Villegas1999}, \cite{Zudilin2014}, and \cite{Brunault2016}.

In a series of papers \cite{Bertin2006, Bertin2008}, Bertin derived similar identities between the Mahler measure of a polynomial defining an elliptic K3 surface and the $L$-value of the associated newform of weight $3$ at $s=3$. She studies families of K3 surfaces defined by families of polynomials $P_k(X,Y,Z) = X+\frac{1}{X}+Y+\frac{1}{Y}+Z+\frac{1}{Z} -k$ and another family $Q_k$. One crucial ingredient in her proof is that the derivatives of $m(P_k)$ and $m(Q_k)$ with respect to $k$ are periods of the corresponding K3 surfaces and hence satisfy a Picard-Fuchs equation. Using this fact she derives an expression of $m(P_k)$ and $m(Q_k)$ as Eisenstein-Kronecker series that she can then relate to $L$-values for special values of $k$ such as
\begin{equation}\label{eqn:Bertins result}
m(P_2) = 4\Lambda(f_8,3),
\end{equation} 
where $f_8$ is the unique newform in $\mathcal{S}_3(\Gamma_1(8))$. A similar approach was adapted in \cite{BertinFeaverFuselierLalin2013} and \cite{Samart2013}, where several more relations between Mahler measures and $L$-values are derived.

In this article we develop a new method for establishing identities such as \eqref{eqn:Bertins result} for three-variable polynomials that define elliptic modular surfaces. It is based on a modification of Deninger's approach in \cite{Deninger1997} that, in favourable cases, allows us to express the Mahler measure of a polynomial as a Deligne period of the corresponding variety; in the $3$-variable case as the integral of a Deligne $2$-form $\eta$ over a $2$-cycle $D$. Our approach relies on explicitly identifying $\eta$ in terms of Eisenstein symbols (see \S\ref{section:Eisenstein symbols}) and $D$ in terms of modular symbols (see \S\ref{section:Shokurov cycles}) and then applying results of \cite{brunault:regRZ} to evaluate the integral. The strength of our method lies in the fact that once the cycle $D$ is identified, most other steps can be performed algorithmically. 

In the following theorem we collect our results. For a congruence subgroup $\Gamma$, the surface $E(\Gamma)$ is the universal elliptic curve over the (open) modular curve $Y(\Gamma)$, see \S\ref{subsection:Universal elliptic curves}. Each equation is preceded by $\Gamma$, signifying that the given Laurent polynomial defines a model for $E(\Gamma)$. While the first of our results was found by Bertin \cite{Bertin2006} with a different method, the rest, to our knowledge, are new.

\begin{thm}\label{thm:Summary} For $N \in \{8,12,16\}$, let $f_N$ denote the unique newform in $\mathcal{S}_3(\Gamma_1(N))$ with rational Fourier coefficients. We have
\begin{align*}
\Gamma_1(8):~ &m(X+\frac{1}{X}+Y+\frac{1}{Y}+Z+\frac{1}{Z} -2) = 4\Lambda(f_8,0),\\
\Gamma_1(8):~ &m((X-1)^2(Y-1)^2 - \frac{(Z-1)^4}{Z^2}XY) = 8\Lambda(f_8,0),\\
\Gamma_1(4)\cap\Gamma_0(8):~ &m(X+\frac{1}{X}+Y+\frac{1}{Y}-4Z-4) = 7\frac{\zeta(3)}{\pi^2}+\log(2),\\
\Gamma_1(6):~ &m((X+Y+1)\left(\frac{1}{X}+\frac{1}{Y}+1\right)-Z) = 7\frac{\zeta(3)}{\pi^2},\\
\Gamma_1(6)\cap\Gamma(2):~ &m((X+Y+1)\left(\frac{1}{X}+\frac{1}{Y}+1\right)-2\left(Z+\frac{1}{Z}\right)-5) = \frac32\Lambda(f_{12},0) + \frac{\log(2)}{2},\\
\Gamma(4):~ &m((X+1)^2 (Y+1)^2 - 2 \frac{(Z+1)^4}{Z^3+Z}XY) = 4\Lambda(f_{16},0).
\end{align*}
\end{thm}
Note that there are no cusp forms of weight $3$ for the congruence subgroups $\Gamma_1(4)\cap\Gamma_0(8)$ and $\Gamma_1(6)$. Accordingly, the Mahler measures we obtain are combinations of $L$-values of Eisenstein series at $s=0$.

In the following sections we will develop our method for general congruence subgroups, accompanying the general considerations with explicit calculations for Bertin's polynomial $P=P_2$. In \S\ref{section:The Deninger cycle} we recall part of Deninger's work, expressing the Mahler measure of a three-variable polynomial as an integral of a Deligne $2$-form $\eta$ over an explicit $2$-cycle $D$ which we call the Deninger cycle. In \S\ref{section:elliptic modular surface} we describe a way to parametrise $V(P)$ by giving an explicit birational map from $E(\Gamma_1(8))$ to $V(P)$. This construction was communicated to us by Odile Lecacheux and we also use it for all the other surfaces we treat. In \S\ref{section:Shokurov cycles} and \S\ref{section:Eisenstein symbols} we describe explicit cycles and motivic cohomology classes on the modular surface $E(\Gamma)$. The Shokurov cycles are closely connected to modular symbols and ultimately we want to express the Deninger cycle $D$ in terms of such cycles. The cohomology classes, cup products of Eisenstein symbols, are natural elements of $H_{\mathcal{M}}^{k+2}(E(\Gamma)^k,\Q(k+2))$ first constructed by Beilinson in the case $k=0$ and extended by Deninger--Scholl and Gealy for general $k$. Section \S\ref{section:Eisenstein symbols} describes how to associate to a Milnor symbol on $E(\Gamma)$ an Eisenstein symbol, which we use to express the form $\eta$ in terms of Eisenstein symbols. In \S\ref{section:Deligne cohomology} we recall some basic facts about Deligne--Beilinson cohomology and the Beilinson regulator map.

Several of Boyd's conjectures were solved by expressing Mahler measures as integrals of Eisenstein symbols over modular symbols using the Rogers--Zudilin method. In the $2$-variable case this reduces to calculating the regulator of two Siegel units in terms of $L$-values of weight $2$ modular forms (see \cite{Brunault2016}). The first author extended this calculation to deal with regulators of higher order Eisenstein symbols. In particular this allows one to express the integral of an Eisenstein symbol over the Shokurov cycle $gX\{g0,g\infty\}$ for any $g\in\SL_2(\Z)$ in terms of $L$-values. Hence we try to decompose the Deninger cycle into a linear combination of cycles of the form $gX\{g0,g\infty\}$. While every Shokurov cycle is homologous to such a linear combination, we encounter several difficulties, coming from the fact that the Shokurov cycles are cycles on a compactification of $E(\Gamma)$, while the Eisenstein symbols do not extend to cohomology classes on the compactification. This means that an integral of an Eisenstein symbol over a Shokurov cycle is not necessarily absolutely convergent. Furthermore, even given absolute convergence, it is not always true that the integral of an Eisenstein symbol over a coboundary vanishes, so integrals over homologous cycles might not be equal. We carefully deal with these technical difficulties in \S\ref{section:Integrating Eisenstein symbols}.

In \S\ref{section:Calculation} we describe in detail how our method works for Bertin's polynomial $P_2$. In the final section \S\ref{section:Other examples} we give conditions on the three-variable polynomial $P$ under which we expect that our method gives a connection between $m(P)$ and an $L$-value. We describe how to construct polynomials satisfying our conditions and derive the new formulas listed in Theorem \ref{thm:Summary}.

The search for models that satisfy our conditions was performed in Magma \cite{Magma} and we give more details on the algorithm in \S\ref{section:Other examples}. Given a model satisfying the three conditions, all other steps described above can be performed algorithmically and we implemented them in Sage \cite{Sage}. In order to check our results and find other examples we used MATLAB \cite{MATLAB} to numerically approximate the Mahler measures.

\textbf{Acknowledgements.} 
The first author thanks Odile Lecacheux for her suggestion to give a modular proof of Bertin's results, and for communicating her method to obtain a modular parametrisation. We also thank her and Marie-Jos\'{e} Bertin for interesting discussions and encouragement throughout this project. The second author\footnote{Supported by EPSRC grant N007360 `Explicit methods for Jacobi forms over number fields' and the DFG Forschergruppe-1920.} also thanks Bartosz Naskr\k{e}cki and Chris W\"uthrich for helpful conversations about elliptic surfaces. Finally, we thank the anonymous referees for careful reading and comments leading to an improvement of the article.

\section{The Deninger cycle}
\label{section:The Deninger cycle}

Let $P\in\C[X^{\pm 1},Y^{\pm 1},Z^{\pm 1}]$ be given by $P(X,Y,Z)=\sum_{i = i_0}^{i_1} a_i(X,Y)Z^i$ with $a_{i_1}(X,Y)\neq 0$ and define $P^*(X,Y)=a_{i_1}(X,Y)$. We denote by $V(P)$ the (reduced) zero locus of $P$ in $\mathbb{G}_m^3$.
Then by Jensen's formula
\begin{equation}\label{jensen}
m(P)-m(P^*) = \frac{1}{(2\pi i)^2}\int_{(T^{2}\times U)\cap V(P)}\log\abs{Z}\frac{dX}{X}\wedge\frac{dY}{Y},
\end{equation}
where $T^2 = \{(X,Y) : |X|=|Y|=1\}$ and $U=\{Z : |Z|>1\}$. 

\begin{dfn}
The Deninger 2-cycle $D_P$ is defined by
\begin{equation*}
D_P = (T^2 \times U) \cap V(P).
\end{equation*}
\end{dfn}

Note that $D_P$ carries a natural orientation arising from the natural orientation on $T^2$.

We now focus on the polynomial $P(X,Y,Z)=X+\frac{1}{X}+Y+\frac{1}{Y}+Z+\frac{1}{Z}-2$, for which we give a more explicit description of the Deninger cycle $D_P$.

\begin{lem}\label{lem:Deninger cycle explicit}
\begin{align*}
D_P =
\{&(e^{i\phi},e^{i\psi},Z(\phi,\psi)): \cos\phi+\cos\psi<0,\\
&\quad Z(\phi,\psi)=1-\cos\phi-\cos\psi+\sqrt{(1-\cos\phi-\cos\psi)^2-1}\}.
\end{align*}
\end{lem}
\begin{proof}
Let $(X,Y,Z)\in D_P$. Since $(X,Y)\in T^2$ we have
\[
P(X,Y,Z) = 2\Re (X) +2\Re (Y) +Z+\frac{1}{Z} -2. 
\]
The equation $P(X,Y,Z)=0$ is then equivalent to
\[
Z + \frac{1}{Z} = 2(1-\Re (X)-\Re (Y)).
\]
In particular $Z+\frac{1}{Z}\in \R$ which implies $Z\in\R\cup S^1$. Since we assume $|Z|>1$, 
we must have $Z\in(-\infty,-1) \cup (1,+\infty)$. Therefore $|Z+\frac{1}{Z}|> 2$, which implies $\Re (X)+\Re (Y) < 0$ and $Z >1$.
\end{proof}

The region where $\cos \phi+\cos \psi<0$ is (up to translations by $2\pi $) equal to the interior of the square with the corners $(0,\pi),(\pi,2\pi),(\pi,0),(2\pi,\pi)$ and $Z$ varies between $1$ and $3+2\sqrt{2}$.

\begin{lem}\label{fibration DP}
The map $(X,Y,Z) \mapsto Z$ defines a fibration of $D_P \backslash \{(-1,-1,3+2\sqrt{2})\}$ above the interval $(1,3+2\sqrt{2})$. The fibres are homeomorphic to $S^1$ and are given by the implicit equation
\begin{equation*}
\cos\phi + \cos\psi = 1-\frac12 \left(Z+\frac{1}{Z}\right).
\end{equation*}
\end{lem}

Note that for $Z_0 \in (1,3+2\sqrt{2})$, we have $Z_0+\frac{1}{Z_0}-2 \in (0,4)$ and the projective closure of the curve $P(X,Y,Z_0)=0$ is an elliptic curve.

%General description for polynomials in 3 variables
\section{The elliptic modular surface $E_1(8)$}\label{section:elliptic modular surface}

By the work of Bertin-Lecacheux \cite[\S 7.1]{Bertin-Lecacheux}, the elliptic modular surface $E_1(8)$ associated to the group $\Gamma_1(8)$ is a desingularisation of the surface
\begin{equation*}
P(X,Y,Z) = X+\frac{1}{X}+Y+\frac{1}{Y}+Z+\frac{1}{Z}-2=0.
\end{equation*}
The modular curve $Y_1(8)$ has genus $0$ and the structural morphism $E_1(8) \to Y_1(8)$ is given by $(X,Y,Z) \mapsto Z$.

We now give  an explicit parametrisation of the surface $V(P)$ in terms of elliptic functions following a method of Lecacheux \cite{Lecacheux:unpublished}. The birational transformation
\begin{equation*}
U=-XY, \qquad V = -Y(XY+1)
\end{equation*}
puts $V(P)$ into the Weierstra\ss{} form
\begin{equation}\label{eqn:weierstrass form E18}
V^2+\left(Z+\frac{1}{Z}-2\right) UV = U(U-1)^2.
\end{equation}
Let $A : Y_1(8) \to E_1(8)$ be the canonical section of order $8$ on the universal elliptic curve. It has coordinates
\begin{equation*}
(U_A,V_A)=(Z,Z-1)
\end{equation*}
and its multiples are given by
\begin{equation*}
(U_{2A},V_{2A}) = (1,0), \qquad (U_{4A},V_{4A})=(0,0).
\end{equation*}
Let $M=(U,V)$ be a section of $E_1(8)$. We compute explicitly
\begin{align*}
U(M+2A)  & = 1 - \frac{\left(Z+\frac{1}{Z}-2\right) V}{(U-1)^2}\\
V(M+2A) & = - \frac{\left(Z+\frac{1}{Z}-2\right)^2 V}{(U-1)^3} + \frac{Z+\frac{1}{Z}-2}{U-1}
\end{align*}
In particular
\begin{equation*}
U(M+2A)-U_{2A} = - \frac{(Z-1)^2}{Z} \frac{V}{(U-1)^2}.
\end{equation*}
It follows that
\begin{align*}
X & = - \frac{U(U-1)}{V} = \frac{(Z-1)^2}{Z} \cdot \frac{U-U_{4A}}{(U(M+2A)-U_{2A})(U-U_{2A})}\\
Y &= \frac{V}{U-1} = -\frac{Z}{(Z-1)^2} \cdot (U(M+2A)-U_{2A})(U-U_{2A})
\end{align*}
Since the fibre of $E_1(8)(\C)$ above $\tau \in \h$ is isomorphic to $\C/(\Z+\tau\Z)$, there exists a parametrisation of the form
\begin{align*}
U & = u^2 \wp_\tau(z) + r\\
V & = u^3 \left(\frac{\wp'_\tau(z)}{2}\right) + u^2 s \wp_\tau(z) + t
\end{align*}
where $\wp_\tau$ is the Weierstra\ss{} $\wp$-function, and $r,s,t,u$ depend only on $\tau$. In particular, for any points $M_1,M_2,M_3,M_4$ in the fibre of $E_1(8)(\C)$ above $\tau$, we have
\begin{equation*}
\frac{U(M_1)-U(M_2)}{U(M_3)-U(M_4)} = \frac{\wp_\tau(z_1)-\wp_\tau(z_2)}{\wp_\tau(z_3)-\wp_\tau(z_4)}
\end{equation*}
where $M_i$ corresponds to $z_i \in \C/(\Z+\tau\Z)$.

Since the section $A$ corresponds to $z=1/8$, we finally obtain $X,Y,Z$ in terms of $\wp$
\begin{align*}
X &= \frac{(\wp_\tau(1/8)-\wp_\tau(1/4))^2 (\wp_\tau(z)-\wp_\tau(1/2))}{(\wp_\tau(1/8)-\wp_\tau(1/2))(\wp_\tau(z+1/4)-\wp_\tau(1/4))(\wp_\tau(z)-\wp_\tau(1/4))},\\
Y&= -\frac{(\wp_\tau(1/8)-\wp_\tau(1/2)) (\wp_\tau(z+1/4)-\wp_\tau(1/4))(\wp_\tau(z)-\wp_\tau(1/4))}{(\wp_\tau(1/8)-\wp_\tau(1/4))^2(\wp_\tau(1/4)-\wp_\tau(1/2))}, \\
Z&= \frac{\wp_\tau(1/8)-\wp_\tau(1/2)}{\wp_\tau(1/4)-\wp_\tau(1/2)}.
\end{align*}

The elliptic involution on $E_1(8)$ is given by $(X,Y,Z) \mapsto (Y,X,Z)$ and thus $Y(\tau,z)=X(\tau,-z)$. Note also that $Z(\tau)$ is non-vanishing on the upper half-plane, so that $Z$ is a modular unit on $Y_1(8)$.

%Lecacheux's parametrisation.

\begin{lem}\label{lem:lecacheux param}
The Lecacheux parametrisation identifies $E_1(8) \setminus \langle 2A \rangle$ with the open subset of $V(P) \subset \mathbb{G}_m^3$ defined by the condition $Z \notin \{\pm 1, 3\pm 2\sqrt{2}\}$. In particular, we may identify the Deninger cycle $D_P \backslash \{(-1,-1,3+2\sqrt{2})\}$ with a cycle on $E_1(8)(\C)$.
\end{lem}

\begin{proof}
The above construction gives a map $\phi : E_1(8) \setminus \langle 2A \rangle \to V(P) \subset \mathbb{G}_m^3$ which is easily seen to be injective. Since the cusps of $X_1(8)$ are given by the $Z$-values $\{\infty,0,\pm 1, 3\pm 2\sqrt{2}\}$, the image of $\phi$ is contained in the open subset defined by $Z \notin \{\pm 1, 3\pm 2\sqrt{2}\}$. Let us check surjectivity. For the excluded values of $Z$ we have $k=Z+1/Z-2 \in \{0,\pm 4\}$ and it suffices to check that for $k \neq 0,\pm 4$, the curve $\{X+1/X+Y+1/Y+k=0\} \subset \mathbb{G}_m^2$ is an elliptic curve deprived from the subgroup generated by a point of order 4.
\end{proof}

\section{Shokurov cycles}
\label{section:Shokurov cycles}
In this section, we fix an integer $N \geq 4$.

\subsection{Modular curves}\label{subsection:modular curves}

We recall here some basic facts on modular curves following \cite[\S 1]{Kato}.

Let $Y_1(N)$ be the modular curve over $\Q$ classifying pairs $(E,P)$ where $E$ is an elliptic curve and $P$ is a section of exact order $N$ on $E$. We have an isomorphism $Y_1(N)(\C) \cong \Gamma_1(N) \backslash \h$.

Let $Y(N)$ be the modular curve over $\Q$ classifying triples $(E,e_1,e_2)$ where $E$ is an elliptic curve and $(e_1,e_2)$ is a $\Z/N\Z$-basis of $E[N]$. The group $G=\GL_2(\Z/N\Z)$ acts from the left on $Y(N)$ by the rule $g \cdot (E,e_1,e_2) = (E,e'_1,e'_2)$ with
\begin{equation*}
\begin{pmatrix} e'_1 \\ e'_2 \end{pmatrix} = g \begin{pmatrix} e_1 \\ e_2 \end{pmatrix}.
\end{equation*}
The canonical degeneracy map $\pi : Y(N) \to Y_1(N)$ given by $(E,e_1,e_2) \mapsto (E,e_2)$ induces an isomorphism $Y_1(N) \cong \Gamma \backslash Y(N)$ where
\begin{equation*}
\Gamma = \left\{ \begin{pmatrix} * & * \\ 0 & 1 \end{pmatrix}\right\} \subset \GL_2(\Z/N\Z).
\end{equation*}

Following \cite[3.4]{deninger:extensions}, the complex points of $Y(N)$ can be described as
\begin{equation*}
Y(N)(\C) \cong \SL_2(\Z) \backslash (\h \times G)
\end{equation*}
and the action of $G$ is given by $g \cdot (\tau,h) = (\tau, h g^T)$. The degeneracy map $\pi : Y(N)(\C) \to Y_1(N)(\C)$ is determined by
\begin{equation*}
\pi \left( \tau, \begin{pmatrix} 0 & -1 \\ a & 0 \end{pmatrix} \right) = [\tau] \qquad (\tau \in \h, a \in (\Z/N\Z)^\times).
\end{equation*}

\subsection{Universal elliptic curves}
\label{subsection:Universal elliptic curves}
Let us describe the complex points of the universal elliptic curve $E_1(N)$ over $Y_1(N)$. The semi-direct product $\Z^2 \rtimes \SL_2(\Z)$ acts on $\h \times \C$ by the rules
\begin{align*}
\begin{pmatrix} m \\ n \end{pmatrix} \cdot (\tau,z) & = (\tau,z+m-n\tau), \\
\begin{pmatrix} a & b \\ c & d \end{pmatrix} \cdot (\tau,z) & = \Bigl(\frac{a \tau+b}{c\tau+d}, \frac{z}{c\tau+d}\Bigr).
\end{align*}
We then have an isomorphism
\begin{equation*}
E_1(N)(\C) \cong (\Z^2 \rtimes \Gamma_1(N)) \backslash (\mathcal{H} \times \C)
\end{equation*}
and the structural morphism $p_1 : E_1(N)(\C) \to Y_1(N)(\C)$ is given by $p_1([\tau,z]) = [\tau]$.

Let us now describe the complex points of the universal elliptic curve $E(N)$ over $Y(N)$ following \cite[3.4]{deninger:extensions}. We have
\begin{equation*}
E(N)(\C) \cong (\Z^2 \rtimes \SL_2(\Z)) \backslash (\h \times \C \times G)
\end{equation*}
where $\SL_2(\Z)$ acts by left multiplication on $G$. The group $G$ acts from the left on $E(N)$ by the rule $g \cdot (\tau,z,h) = (\tau,z,h g^T)$. The structural morphism $p : E(N)(\C) \to Y(N)(\C)$ is the obvious one, and we have a commutative diagram
\begin{equation}\label{EN E1N}
\begin{tikzcd}
E(N) \arrow{r}{p} \arrow{d}[swap]{\tilde{\pi}} & Y(N) \arrow{d}{\pi} \\
E_1(N) \arrow{r}{p_1} & Y_1(N)
\end{tikzcd}
\end{equation}
where the map $\tilde{\pi}$ is given by
\begin{equation}\label{formula pitilde}
\tilde{\pi} \left( \tau, z, \begin{pmatrix} 0 & -1 \\ a & 0 \end{pmatrix} \right) = [\tau,z] \qquad (\tau \in \h, z \in \C, a \in (\Z/N\Z)^\times).
\end{equation}
Note that the square (\ref{EN E1N}) is cartesian: the universal elliptic curve $E(N)$ can be identified with the base change of $E_1(N)$ to $Y(N)$.

\subsection{Shokurov cycles}

In this subsection, we define the Shokurov cycles on $E_1(N)(\C)$ and $E(N)(\C)$. These cycles were first studied by Shokurov \cite{Shokurov1980} and are at the foundation of the theory of modular symbols.

Let $\Z[X,Y]_1$ be the space of homogeneous polynomials in $X$ and $Y$ of degree $1$. Let $\sigma$ denote the matrix $\sabcd{0}{-1}{1}{0}$.

\begin{dfn}
Let $P =mX+nY \in \Z[X,Y]_1$ and $\alpha,\beta \in \h \cup \mathbb{P}^1(\Q)$. We define the $2$-cycle $P \{\alpha,\beta\}_1$ on $E_1(N)(\C)$ by
\begin{equation*}
P \{\alpha,\beta\}_1 = \left\{ [\tau,t(m\tau+n)]: \tau \in \{\alpha,\beta\}, t \in [0,1] \right\}
\end{equation*}
where $\{\alpha,\beta\}$ denotes the geodesic from $\alpha$ to $\beta$ in $\h$.

We define the $2$-cycle $P \{\alpha,\beta\}$ on $E(N)(\C)$ by
\begin{equation*}
P \{\alpha,\beta\} = \left\{ [\tau,t(m\tau+n),\sigma]: \tau \in \{\alpha,\beta\}, t \in [0,1] \right\}.
\end{equation*}
\end{dfn}

The relation between Shokurov cycles on $E(N)$ and $E_1(N)$ is given by the following lemma.

\begin{lem}\label{shokurov E E1}
Let $\tilde{\pi}: E(N)(\C) \to E_1(N)(\C)$ denote the canonical map. For every $P \in \Z[X,Y]_1$ and every $\alpha,\beta \in \h \cup \mathbb{P}^1(\Q)$, we have
\begin{equation*}
\tilde{\pi}_* (P\{\alpha,\beta\}) = P\{\alpha,\beta\}_1.
\end{equation*}
\end{lem}

\begin{proof} This follows from (\ref{formula pitilde}).
\end{proof}

The group $\SL_2(\Z)$ acts from the left on $\Z[X,Y]_1$ by $\left(\begin{smallmatrix}a&b\\c&d\end{smallmatrix}\right) P(X,Y) = P(dX - bY, -cX + aY)$ and on $\mathbb{P}^1(\Q)$ by M\"obius transformations.

\begin{lem} \label{action G shokurov}
For all $\gamma\in\SL_2(\Z)$ with reduction $\bar{\gamma} \in \SL_2(\Z/N\Z)$, and for all cycles $P\{\alpha,\beta\}$ on $E(N)(\C)$, we have
\[
\bar{\gamma}_* (P\{\alpha,\beta\}) = \gamma P \{\gamma\alpha,\gamma\beta\}. 
\]
\end{lem}

\begin{proof}
Write $\gamma = \begin{pmatrix} a & b \\ c & d \end{pmatrix}$. By definition, we have
\begin{equation*}
\bar{\gamma}_* (P \{\alpha,\beta\}) = \left\{ [\tau, t P(\tau,1), \sigma \bar{\gamma}^T ] : \tau \in \{\alpha,\beta\}, t \in [0,1]\right\}.
\end{equation*}
Since $\sigma \gamma^T = \gamma^{-1} \sigma$, we get
\begin{equation*}
\bar{\gamma}_* (P \{\alpha,\beta\}) = \left\{ \left[\gamma \tau, \frac{t P(\tau,1)}{c\tau+d}, \sigma\right] : \tau \in \{\alpha,\beta\}, t \in [0,1]\right\}.
\end{equation*}
On the other hand, we have
\begin{equation*}
\gamma P \{\gamma\alpha,\gamma\beta\} = \left\{ \left[\gamma \tau,t \cdot (\gamma P)(\gamma \tau,1),\sigma \right] : \tau \in \{\alpha,\beta\}, t \in [0,1]\right\}
\end{equation*}
and a simple computation shows that $(\gamma P)(\gamma\tau,1) = P(\tau,1)/(c\tau+d)$.
\end{proof}

\begin{lem}\label{lem int sum shokurov}
For any closed $2$-form $\eta$ on $E(N)(\C)$ and any $\tau_1,\tau_2 \in \h$, we have
\begin{equation*}
\int_{(mX+nY)\{\tau_1,\tau_2\}} \eta = m \left(\int_{X\{\tau_1,\tau_2\}} \eta\right) + n \left(\int_{Y\{\tau_1,\tau_2\}} \eta\right) \qquad (m,n \in \Z).
\end{equation*}
\end{lem}

\begin{proof}
For $P \in \Z[X,Y]_1$, consider the natural fibration $p_P : P\{\tau_1,\tau_2\} \to \{\tau_1,\tau_2\}$. For a given $\tau$ on the geodesic $\{\tau_1,\tau_2\}$, the fibre of $p_{mX+nY}$ over $\tau$ is homologous to $m \cdot [p_X^{-1}(\tau)] + n \cdot [p_Y^{-1}(\tau)]$. Integrating over the fibres, this implies that
\begin{equation*}
\int_{p_{mX+nY}} \eta = m \int_{p_X} \eta + n \int_{p_Y} \eta.
\end{equation*}
The Lemma follows.
\end{proof}

\section{Eisenstein symbols}
\label{section:Eisenstein symbols}
The goal of this section is to express the Milnor symbol $\{X,Y,Z\}$ in terms of Eisenstein symbols on the universal elliptic curve $E_1(8)$.

\subsection{Siegel units} \label{siegel units}

In this subsection, we express $Z$ in terms of Siegel units.

Let $N \geq 3$ be an integer. For any $(a,b) \in (\Z/N\Z)^2$, $(a,b) \neq (0,0)$, the Siegel unit $g_{a,b} : \h \to \C^\times$ is defined by the following infinite product (see \cite[\S 1]{Kato}):
\begin{equation}\label{eqn:siegel-infinite-product}
g_{a,b}(\tau) = q^{\frac12 B_2\left(\tilde{a}/N\right)} \prod_{n\geq 0}(1-q^{n+\tilde{a}/N} \zeta_N^b)\prod_{n \geq 1}(1-q^{n-\tilde{a}/N} \zeta_N^{-b}),
\end{equation}
where $q^\alpha=e^{2\pi i \alpha \tau}$, $\zeta_N=e^{2\pi i/N}$, $B_2(X)=X^2-X+\frac16$ is the second Bernoulli polynomial, and $\tilde{a} \in \Z$ is the unique representative of $a$ satisfying $0 \leq \tilde{a} < N$. Note that $g_{-a,-b}=g_{a,b}$ when $a \neq 0$, and $g_{0,-b}=-\zeta_N^{-b} g_{0,b}$.

Under the action of $\SL_2(\Z)$ on $\h$, the Siegel units transform by $\gamma^* g_{a,b} = \zeta g_{(a,b) \gamma}$ for some root of unity $\zeta$ \cite[1.6, 1.7]{Kato}. This root of unity is given by the following proposition.

\begin{prop}\label{prop:siegel-transformation}
Let $a,b \in \Z/N\Z$, $(a,b) \neq (0,0)$, and let $\gamma \in \SL_2(\Z)$.
\begin{enumerate}
\item
If $\gamma=\sabcd {\pm 1}k0{\pm 1}$ for $k\in\Z$ then
\[
\gamma^* g_{a,b}(\tau) = g_{a,b}(\tau\pm k) = e^{\pm\pi i k B_2(\tilde{a}/N)} g_{a,\pm ka+b}(\tau).
\]
\item
Assume $\gamma = \sabcd cedf$ with $d>0$. Denote $(a',b') = (a,b)\gamma$. The root of unity $(\gamma^* g_{a,b})/g_{a',b'}$ is determined by  
\begin{multline*}
\arg \left(\frac{\gamma^* g_{a,b}}{g_{a',b'}}\right) \equiv \pi\left[ \frac{c}{d}B_2\left(\frac{\tilde{a}}{N}\right) + \frac{f}{d}B_2\left(\frac{\tilde{a}'}{N}\right) + \delta(\tilde{a})\mathcal{P}_1\left(\frac{b}{N}\right) - \delta(\tilde{a}')\mathcal{P}_1\left(\frac{b'}{N}\right)\right.
 \\
\left.
 - 2 \sum_{k=1}^d \mathcal{P}_1\left(\frac{1}{d}(k-\frac{\tilde{a}}{N})\right) \mathcal{P}_1\left(\frac{c}{d}(k-\frac{\tilde{a}}{N})-\frac{b}{N}\right)
\right]
\mod 2\pi
\end{multline*}
where $\mathcal{P}_1(x) = B_1(\{x\}) = \{x\}-\frac12$ if $x\notin\Z$ and $\mathcal{P}_1(x) = 0$ otherwise. 
\end{enumerate}
\end{prop}
\begin{proof}
The first statement follows from the infinite product expansion of $g_{a,b}$. The second follows from an explicit calculation of the limit of $\arg g_{a,b}(c/d+it)$ when $t \to 0$. It is carried out in detail in \cite[Proof of Thm 13, pp.135-143]{Siegel:advanced_analytic_nt}\footnote{To compare Siegel's notation with ours set $u=\tilde{a}/N, v=-\tilde{b}/N$. Then $\phi(u,v,z) = -ie^{-\pi i v(u-1)} g_{a,b}$.}. Following Siegel's calculation we find
\begin{align*}
\lim_{t\to 0}\frac{\arg g_{a,b}(c/d+it)}{\pi} &\equiv \frac{c}{d}B_2\left(\frac{\tilde{a}}{N}\right) +\delta(\tilde{a})\mathcal{P}_1\left(\frac{b}{N}\right)
\\ &\quad - 2\sum_{k=1}^d \mathcal{P}_1\left(\frac{1}{d}(k-\frac{\tilde{a}}{N})\right) \mathcal{P}_1\left(\frac{c}{d}(k-\frac{\tilde{a}}{N})-\frac{b}{N}\right)\mod 2\pi.
\end{align*}
Since $(\gamma^* g_{a,b})/g_{a',b'}$ is constant,
\[
g_{a,b}(c/d+it) = \frac{g_{a,b}(\gamma\gamma^{-1}(c/d+it))}{g_{a',b'}(\gamma^{-1}(c/d+it))}
g_{a',b'}(\gamma^{-1}(c/d+it))=
\left(\frac{\gamma^* g_{a,b}}{g_{a',b'}}\right)g_{a',b'}\left(\frac{i}{d^2 t}-\frac{f}{d}\right),
\]
so 
\begin{align*}
\lim_{t\to 0}\arg g_{a,b}(c/d+it) &\equiv  \arg \left(\frac{\gamma^* g_{a,b}}{g_{a',b'}}\right) + \lim_{t\to 0}\arg g_{a',b'}\left(\frac{i}{d^2 t}-\frac{f}{d}\right) \mod 2\pi\\
&\equiv
 \arg \left(\frac{\gamma^* g_{a,b}}{g_{a',b'}}\right) + \arg (1-\delta(\tilde{a}')\zeta_N^{b'}) - \pi \frac{f}{d} B_2(\frac{\tilde{a}'}{N})\mod 2\pi.
\end{align*}
The second term equals $\delta(\tilde{a}')\pi \mathcal{P}_1(b'/N)$.
\end{proof}

Using Proposition \ref{prop:siegel-transformation} we can evaluate any quotient of Siegel units at any cusp.
\begin{prop}\label{prop:evaluate siegel unit}
Let $F=\prod_{i=1}^k g_{a_i,b_i}^{c_i}$ with $c_i\in\Z$ be a quotient of Siegel units. Let $\gamma=\sabcd cedf\in\SL_2(\Z)$ and set $(a_i',b_i') = (a_i,b_i)\gamma$. The quantity $\ord_{c/d} F := \frac{1}{2}\sum_{i=1}^k c_i B_2(\tilde{a}'_i/N)$ does not depend on a choice of $\gamma$. If $\ord_{c/d} F$ is positive, then $F(c/d)=\lim_{t\to 0} F(c/d+it) = 0$ and if $\ord_{c/d} F$ is negative the absolute value of $F(c/d+it)$ diverges to $\infty$. If $\ord_{c/d} F=0$, then
\[
F(c/d) := \lim_{t\to 0} F(c/d+it) =
 \prod_{i=1}^k \left(\frac{\gamma^* g_{a_i,b_i}}{g_{a_i',b_i'}}\right)^{c_i} (1-\delta(\tilde{a_i}')\zeta_N^{b_i'})^{c_i}.
\]
\end{prop}
\begin{proof}
Writing $g_{a_i,b_i} = (\gamma^* g_{a_i,b_i})/g_{a'_i,b'_i} \cdot (g_{a'_i,b'_i} \circ \gamma^{-1})$ as in the proof of Proposition \ref{prop:siegel-transformation}, we get
\begin{align*}
\lim_{t\to 0} F(c/d+it) &= \prod_{i=1}^k \left(\frac{\gamma^* g_{a_i,b_i}}{g_{a_i',b_i'}}\right)^{c_i} \lim_{t\to 0} \prod_{i=1}^k g^{c_i}_{a_i',b_i'}\left(\frac{i}{d^2 t} - \frac{f}{d}\right).
\end{align*}
The last limit equals
\[
\lim_{t\to \infty}\prod_{i=1}^k g^{c_i}_{a_i',b_i'}\left(it - \frac{f}{d}\right) = \left(\prod_{i=1}^k (1-\delta(\tilde{a_i}')\zeta_N^{b_i'})^{c_i}\right) \lim_{t\to\infty} e^{2\pi(- t - i \frac{f}{d})\ord_{c/d}F}.\]
The Proposition follows.
\end{proof}
We denote by $\mathcal{O}(Y(N))^\times$ the group of modular units on the modular curve $Y(N)$. After tensoring with $\Q$, the Siegel units are modular units of level $\Gamma(N)$, so they can be viewed as elements of $\mathcal{O}(Y(N))^\times \otimes \Q$. Inside this $\Q$-vector space, the Siegel units satisfy the transformation law $\gamma^* g_{a,b} = g_{(a,b)\gamma}$ for any $\gamma \in \GL_2(\Z/N\Z)$. Note that the Siegel units $g_{0,b}$ with $b \in \Z/N\Z$, $b \neq 0$ are invariants under the group $\left\{\left(\begin{smallmatrix} * & * \\ 0 & 1 \end{smallmatrix}\right)\right\}$ and therefore descend to modular units on $Y_1(N)$.

\begin{prop}\label{proZ1}
We have the following identity
\begin{equation} \label{Z1}
Z(\tau) = (3+2\sqrt{2}) \prod_{n \geq 1} \frac{(1-q^n \zeta_8^3)^2 (1-q^n \zeta_8^{-3})^2}{(1-q^n \zeta_8)^2 (1-q^n \zeta_8^{-1})^2}=-i\left(\frac{g_{0,3}}{g_{0,1}}\right)^2.
\end{equation}
\end{prop}

\begin{proof}
Note that
\begin{equation*}
Z(\tau) = \frac{\wp_\tau(1/8)-\wp_\tau(1/2)}{\wp_\tau(1/4)-\wp_\tau(1/2)}
\end{equation*}
is a Weierstrass unit in the terminology of Kubert-Lang \cite[Chap. 2, \S 6]{Kubert-Lang}, and Weierstrass units can be expressed in terms of Siegel units \cite[Prop. 2.2]{brunault:K2}. To get the explicit expression we write $Z$ in terms of the Weierstrass $\sigma$-function, using that for $z_1,z_2\in\C$ we have \cite[Corollary 5.6(a)]{Silverman2}
\[
\wp_\tau(z_1)-\wp_\tau(z_2) = - \frac{\sigma_\tau(z_1+z_2)\sigma_\tau(z_1-z_2)}{\sigma_\tau(z_1)^2\sigma_\tau(z_2)^2}.
\] 
We get
\[
Z(\tau) = \frac{\sigma_\tau(1/4) \sigma_\tau(3/8) \sigma_\tau(5/8)}{\sigma_\tau(3/4) \sigma_\tau(1/8)^2}.
\]
We then use the product expansion for $\sigma_\tau(z)$ \cite[Theorem 6.4]{Silverman2} and the identity $3+2\sqrt{2}=-i(1-\zeta_8^3)^2/(1-\zeta_8)^2$ to confirm (\ref{Z1}).
\end{proof}

\subsection{The Eisenstein symbol}

Let $X(N)$ be the compactification of $Y(N)$, and let $X(N)^\infty$ be the scheme of cusps of $X(N)$. We have a bijection
\begin{align*}
\left\{\pm \begin{pmatrix} * & * \\ 0 & 1 \end{pmatrix} \right\} \backslash \GL_2(\Z/N\Z) & \xrightarrow{\cong} X(N)^\infty \\
[g] & \mapsto g^T \cdot \infty. 
\end{align*}
The vector space
\begin{equation*}
V_N^\pm = \left\{f : \GL_2(\Z/N\Z) \to \Q : f \left(\begin{pmatrix} * & * \\ 0 & 1 \end{pmatrix} g \right) = f(g) = \pm f(-g), g \in \GL_2(\Z/N\Z) \right\}
\end{equation*}
is non canonically isomorphic to the $\Q$-vector space with basis $X(N)^\infty$.

Let $k \geq 0$ be an integer, and let $E(N)^k$ be the $k$-fold fibre product of $E(N)$ over $Y(N)$. We denote motivic cohomology by $H^{\cdot}_{\mathcal{M}}$. Beilinson \cite{Beilinson2} constructed a residue map
\begin{equation*}
\Res^k_{\mathcal{M}} : H^{k+1}_{\mathcal{M}}(E(N)^k, \Q(k+1)) \to V_N^{(-1)^k}
\end{equation*}
generalising the divisor map $H^1_{\mathcal{M}}(Y(N),\Q(1)) \to V_N^+$ in the case $k=0$. In the case $k \geq 1$, Beilinson also constructed a canonical right inverse of $\Res^k_{\mathcal{M}}$, the \emph{Eisenstein map}
\begin{equation*}
E^k_{\mathcal{M}} : V_N^{(-1)^k} \to H^{k+1}_{\mathcal{M}}(E(N)^k, \Q(k+1)).
\end{equation*}

\begin{dfn}
For any integer $k \geq 0$, the horospherical map
\begin{equation*}
\lambda^k : \Q[(\Z/N\Z)^2] \to V_N^{(-1)^k}
\end{equation*}
is defined by
\begin{equation*}
\lambda^k(\phi)(g) = \sum_{x_1,x_2 \in \Z/N\Z} \phi \left(g^{-1} \begin{pmatrix} x_1 \\ x_2 \end{pmatrix} \right) B_{k+2} \left(\left\{\frac{x_2}{N}\right\} \right) \qquad (g \in \GL_2(\Z/N\Z))
\end{equation*}
where $B_{k+2}$ is the $(k+2)$-th Bernoulli polynomial, and $\{t\} = t-\lfloor t \rfloor$ is the fractional part of $t$.
\end{dfn}

\begin{dfn}\label{def Eis}
For any integer $k \geq 1$ and any $u \in (\Z/N\Z)^2$, the \emph{Eisenstein symbol}
\begin{equation*}
\Eis^k(u) \in H^{k+1}_{\mathcal{M}}(E(N)^k,\Q(k+1))
\end{equation*}
is defined by $\Eis^k(u) = E^k_{\mathcal{M}} \circ \lambda^k (\phi_u)$, where $\phi_u = [u]$ is the characteristic function of $\{u\}$.
\end{dfn}

Note that $\Eis^k(-u) = (-1)^k \Eis^k(u)$ for any $u \in (\Z/N\Z)^2$.

In the case $k=0$, the image of the divisor map $\Res^0_\mathcal{M}$ is the space $(V_N^+)^0$ of degree 0 divisors on $X(N)^\infty$ by the Manin-Drinfeld theorem. It can be shown that the Siegel units provide a canonical right inverse $E^0_{\mathcal{M}} : (V_N^+)^0 \to H^1(Y(N),\Q(1))$ of $\Res^0_\mathcal{M}$. More precisely, the horospherical map $\lambda^0$ induces a map $\Q[(\Z/N\Z)^2 \backslash \{0\}] \to (V_N^+)^0$, and we have
\begin{equation*}
E^0_{\mathcal{M}} \circ \lambda^0(\phi_u) = g_u \otimes \frac{2}{N} \qquad (u \in (\Z/N\Z)^2, u \neq (0,0)).
\end{equation*}
In analogy with Definition \ref{def Eis}, we put $\Eis^0(u) = g_u \otimes (2/N)$ for any $u \in (\Z/N\Z)^2$, $u \neq (0,0)$. Proposition \ref{proZ1} can thus be rewritten as follows.

\begin{prop}\label{proZ2}
Let $\pi : Y(8) \to Y_1(8)$ be the canonical projection map. Then
\begin{equation*}
\pi^*(Z) = 8 (\Eis^0(0,3)- \Eis^0(0,1)).
\end{equation*}
\end{prop}

\subsection{Definition of the Eisenstein symbol in the case $k=1$}

Let us now recall the construction of the Eisenstein map in the case $k=1$ following Deninger--Scholl \cite{Deninger-Scholl}. The group $(\Z/N\Z)^2 \cong E(N)[N]$ acts by translation on $E(N)$, and the group $\{\pm 1\}$ acts on $E(N)$ by means of the elliptic involution $\iota : E(N) \to E(N)$. We thus get an action of $H:=(\Z/N\Z)^2 \rtimes \{\pm 1\}$ on $E(N)$. Let $\varepsilon : H \to \{\pm 1\}$ be the canonical projection.

\begin{dfn}
For any $\Q[H]$-module $M$, we denote by $M^\varepsilon$ the $\varepsilon$-eigenspace of $M$, and we denote by $\Pi_\varepsilon : M \to M^\varepsilon$ the projector defined by
\begin{equation*}
\Pi_\varepsilon(m) = \frac{1}{|H|} \sum_{h \in H} \varepsilon(h) h \cdot m \qquad (m \in M).
\end{equation*}
\end{dfn}

Let $U_N = E(N) - E(N)[N]$ be the complement of the $N$-torsion subgroup of $E(N)$. It is stable by $H$. The following theorem is due to Beilinson \cite[Theorem 3.1.1]{Beilinson2}.

\begin{thm}\label{thm H2E}
\begin{enumerate}
\item The group $(\Z/N\Z)^2$ acts trivially on $H^2_{\mathcal{M}}(E(N),\Q(2))$. In particular
\begin{equation*}
H^2_{\mathcal{M}}(E(N),\Q(2))^\varepsilon = H^2_{\mathcal{M}}(E(N),\Q(2))^-,
\end{equation*}
where $(\cdot)^-$ denotes the subspace of anti-invariants under the elliptic involution.
\item The restriction map from $E(N)$ to $U_N$ induces an isomorphism
\begin{equation*}
H^2_{\mathcal{M}}(E(N),\Q(2))^- \xrightarrow{\cong} H^2_{\mathcal{M}}(U_N,\Q(2))^\varepsilon.
\end{equation*}
\end{enumerate}
\end{thm}

Consider now the map given by the cup-product
\begin{align*}
\alpha : \mathcal{O}(U_N)^\times \otimes \mathcal{O}(U_N)^\times & \to H^2_\mathcal{M}(U_N,\Q(2))\\
g_0 \otimes g_1 & \mapsto \{ \iota^* g_0, g_1 \}.
\end{align*}
By \cite[Lemma 4.8]{Deninger-Scholl}, we have an exact sequence
\begin{equation*}
0 \to \mathcal{O}(Y(N))^\times \otimes \Q \to \mathcal{O}(U_N)^\times \otimes \Q \xrightarrow{\dv} \Q[(\Z/N\Z)^2]^0 \to 0
\end{equation*}
where $\dv$ is the divisor map, and $\Q[(\Z/N\Z)^2]^0$ is the space of degree 0 divisors on $(\Z/N\Z)^2$. Let $\mu$ denote the product in the group algebra $\Q[(\Z/N\Z)^2]$. We have a diagram
\begin{equation}\label{diag alpha}
\begin{tikzcd}
\mathcal{O}(U_N)^\times \otimes \mathcal{O}(U_N)^\times \otimes \Q \arrow{r}{\alpha} \arrow{d}{\dv \otimes \dv} & H^2_\mathcal{M}(U_N,\Q(2)) \\
\bigotimes^2 \Q[(\Z/N\Z)^2]^0 \arrow{r}{\mu} & \Q[(\Z/N\Z)^2]^0.
\end{tikzcd}
\end{equation}

\begin{lem}\label{Pi alpha factors}
The map $\Pi_\varepsilon \circ \alpha$ factors through the surjective map $\dv \otimes \dv$.
\end{lem}

\begin{proof}
It suffices to show that for any modular unit $u \in \mathcal{O}(Y(N))^\times$ and any $g \in \mathcal{O}(U_N)^\times$, we have $\Pi_\varepsilon(\{u,g\})=0$. Let us first average $\{u,g\}$ over the translations: we get
\begin{equation*}
\frac{1}{N^2} \sum_{a \in (\Z/N\Z)^2} t_a^* \{u,g\} = \frac{1}{N^2} \sum_{a \in (\Z/N\Z)^2} \{u, t_a^* g\} = \frac{1}{N^2} \{u,h\}
\end{equation*}
with $h = \prod_{a \in (Z/N\Z)^2} t_a^* g$. The function $h$ has trivial divisor, so it must come from the base. Now averaging over the group $\{\pm 1\}$, we get $\{u,h\} - \iota^*\{u,h\} = \{u,h\}-\{u,h\}=0$.
\end{proof}

For any $a \in (\Z/N\Z)^2$, let $t_a : U_N \to U_N$ be the correponding translation. We endow $\mathcal{O}(U_N)^\times \otimes \mathcal{O}(U_N)^\times$ with the action of $H$ defined by $a \cdot (g_0 \otimes g_1) = (t_{-a}^* g_0) \otimes (t_a^* g_1)$ for any $a \in (\Z/N\Z)^2$, and $(-1) \cdot (g_0 \otimes g_1) = - (g_1 \otimes g_0)$. We also endow $\bigotimes^2 \Q[(\Z/N\Z)^2]^0$ with the action of $H$ defined by $a \cdot (d_0 \otimes d_1) = ([a] d_0) \otimes ([-a] d_1)$ and $(-1)\cdot (d_0 \otimes d_1) = -(d_1 \otimes d_0)$. Finally, we endow $\Q[(\Z/N\Z)^2]^0$ with the trivial action of $(\Z/N\Z)^2$ and the action of $(-1)$ given by $-\operatorname{id}$. In this way the diagram (\ref{diag alpha}) becomes $H$-equivariant.

\begin{lem}\label{mu iso} The map $\mu$ induces an isomorphism
\begin{equation*}
\begin{tikzcd}
\left(\bigotimes^2 \Q[(\Z/N\Z)^2]^0\right)^\varepsilon \cong \Q[(\Z/N\Z)^2]^0.
\end{tikzcd}
\end{equation*}
\end{lem}

\begin{proof}
Let $I = \Q[(\Z/N\Z)^2]^0$ denote the augmentation ideal of $\Q[(\Z/N\Z)^2]$. Note that with our definition $I^\varepsilon = I$, so that we have a map $\mu : (I \otimes I)^\varepsilon \to I$. We shall now construct an inverse of $\mu$. Consider the divisor $d_0 = N^2 [0] - \sum_{u \in (\Z/N\Z)^2} [u] \in I$, and define $\theta : I  \to (I \otimes I)^\varepsilon$ by $\theta(d) = \Pi_\varepsilon (d \otimes d_0)$. A simple computation shows that $\mu \circ \theta(d) = d d_0 = N^2 d$.

Now let $x = \sum_i d_i \otimes d'_i \in (I \otimes I)^\varepsilon$. We have
\begin{equation*}
\theta \circ \mu(x) = \sum_i \theta(d_i d'_i) = \sum_i \Pi_\varepsilon((d_i d'_i) \otimes d_0).
\end{equation*}
We claim that for any divisors $d,d',d''$ on $(\Z/N\Z)^2$, we have
\begin{equation}\label{Pi dd'd''}
\Pi_\varepsilon((dd') \otimes d'') = \Pi_\varepsilon(d \otimes (d'd'')),
\end{equation}
where we have extended the action of $H$ to $\bigotimes^2 \Q[(\Z/N\Z)^2]$ in the natural way. Indeed, formula (\ref{Pi dd'd''}) is trilinear with respect to $d,d',d''$, and it suffices to check it in the case $d=[u]$, $d'=[u']$ and $d''=[u'']$, where it is obvious. It follows that
\begin{equation*}
\theta \circ \mu(x)= \sum_i \Pi_\varepsilon(d_i \otimes (d'_i d_0)) = N^2 \sum_i \Pi_\varepsilon(d_i \otimes d'_i) = N^2 x.
\end{equation*}
This shows that $(1/N^2) \theta$ is an inverse of $\mu$.
\end{proof}

By considering the $\varepsilon$-eigenspaces in the diagram (\ref{diag alpha}) and using Lemmas \ref{Pi alpha factors} and \ref{mu iso}, we get a unique map $\mathcal{E}^1_{\mathcal{M}} : \Q[(\Z/N\Z)^2]^0 \to H^2_{\mathcal{M}}(E(N),\Q(2))^-$ making the following diagram commute
\begin{equation}\label{diag alpha eps}
\begin{tikzcd}
\mathcal{O}(U_N)^\times \otimes \mathcal{O}(U_N)^\times \otimes \Q \arrow{r}{\Pi_\varepsilon \circ \alpha} \arrow{d}{\dv \otimes \dv} & H^2_\mathcal{M}(E(N),\Q(2))^- \\
\bigotimes^2 \Q[(\Z/N\Z)^2]^0 \arrow{r}{\frac{1}{N^2} \cdot \mu} & \Q[(\Z/N\Z)^2]^0 \arrow{u}{\mathcal{E}^1_\mathcal{M}}.
\end{tikzcd}
\end{equation}

Note that we choose the bottom map to be $(1/N^2)\cdot \mu$ so that $\mathcal{E}^1_\mathcal{M}$ coincides with the map $\mathcal{E}^1_P$ from \cite{Schappacher-Scholl}. To pass from $\mathcal{E}^1_\mathcal{M}$ to the Eisenstein map $E^1_\mathcal{M}$, we use the horospherical map $\lambda^1$.

\begin{thm}\cite{Beilinson2, Schappacher-Scholl}\label{thm E1}
\begin{enumerate}
\item The horospherical map $\lambda^1 : \Q[(\Z/N\Z)^2]^0 \to V_N^-$ is surjective.
\item The map $\mathcal{E}^1_\mathcal{M}$ factors through $\lambda^1$.
\item We have
\begin{equation*}
\Res^1_\mathcal{M} \circ \mathcal{E}^1_\mathcal{M} = \pm \frac{N^3}{3} \lambda^1.
\end{equation*}
\end{enumerate}
\end{thm}

\begin{rem}
Part (1) is proved in \cite[7.5]{Schappacher-Scholl}, and part (2) follows from \cite[Theorem 7.4]{Schappacher-Scholl}. Note that there is a typo in \cite[2.6]{Schappacher-Scholl}: the formula at the top of p. 310 should read $\tilde{\mathbf{B}}_k(\zeta) := N^{k-1} \hat{\mathbf{B}}_{k,N}(\zeta)$. Consequently, the constant in \cite[Proposition 3.2]{Schappacher-Scholl} should be $N/3$ instead of $1/(3N)$, the constant in \cite[Corollary 3.3]{Schappacher-Scholl} should be $N^3/3$ instead of $N/3$, and the constant $C^n_{P,N}$ at the end of \cite[\S 4]{Schappacher-Scholl} should be $C^n_{P,N} = \pm N^{2n+1} (n+1)/(n+2)!$.
\end{rem}

By Theorem \ref{thm E1}(2), there is a commutative diagram
\begin{equation*}
\begin{tikzcd}
\Q[(\Z/N\Z)^2]^0 \arrow{rr}{\mathcal{E}^1_\mathcal{M}} \arrow{dr}{\lambda^1} & & H^2_\mathcal{M}(E(N),\Q(2))^-\\
& V_N^- \arrow{ur}{\overline{\mathcal{E}}^1_\mathcal{M}}
\end{tikzcd}
\end{equation*}
and by Theorem \ref{thm E1}(3), the map $E^1_\mathcal{M} := \pm (3/N^3) \cdot \overline{\mathcal{E}}^1_\mathcal{M}$ is a right inverse of $\Res^1_\mathcal{M}$. Using Definition \ref{def Eis}, we get the following expression of $\mathcal{E}^1_\mathcal{M}$ in terms of the Eisenstein symbol.

\begin{cor}\label{cor Eis1}
We have
\begin{equation*}
\mathcal{E}^1_\mathcal{M}([u]-[0]) = \pm \frac{N^3}{3} \cdot \Eis^1(u) \qquad (u \in (\Z/N\Z)^2).
\end{equation*}
\end{cor}

\subsection{Expressing $\{X,Y\}$ in terms of Eisenstein symbols}

\begin{lem}\label{lem XY}
The Milnor symbol $\{X,Y\}$ defines an element of $H^2_\mathcal{M}(E_1(8),\Q(2))^-$.
\end{lem}

\begin{proof}
Let $t_{2A}$ denote the translation by $2A$ on $E_1(8)$. By \cite[Proof of Corollary 7.1]{Bertin-Lecacheux}, we have $t_{2A}^*(X)=Y$ and $t_{2A}^*(Y)=1/X$, so that $t_{2A}^* \{X,Y\} = \{Y,1/X\} = \{X,Y\}$. We check that the tame symbols of $\{X,Y\}$ at the sections $0$ and $A$ are equal to 1. It follows that $\{X,Y\}$ extends to an element of $H^2_{\mathcal{M}}(E_1(8),\Q(2))$. Since $\iota^* \{X,Y\} = \{Y,X\} = -\{X,Y\}$, we have $\{X,Y\} \in H^2_\mathcal{M}(E_1(8),\Q(2))^-$.
\end{proof}

We denote by $\tilde{\pi}$ the canonical projection map $E(8) \to E_1(8)$.

\begin{lem}\label{div piY}
We have
\begin{align*}
\dv(\tilde{\pi}^* X) & =  - (0,0) - (0,2) + (0,4) + (0,6)\\
\dv(\tilde{\pi}^* Y) & = - (0,0) + (0,2) + (0,4) - (0,6).
\end{align*}
\end{lem}

\begin{proof}
By \cite[\S 7.4]{Bertin-Lecacheux}, the horizontal divisors of $X$ and $Y$ on $E_1(8)$ are given by
\begin{align*}
\dv(X) & = - (0) - (2A) + (4A) + (6A) \\
\dv(Y) & = - (0) + (2A) + (4A) - (6A).
\end{align*}
Note that the universal elliptic curve $E(8)$ over $Y(8)$ is canonically isomorphic to the base change $E_1(8) \times_{Y_1(8)} Y(8)$. The base change of the section $kA : Y_1(8) \to E_1(8)$ is the section $(0,k) : Y(8) \to E(8)$, and the pull-back of the divisor $[kA]$ by $\tilde{\pi}$ is simply the divisor $[(0,k)]$. The lemma follows.
\end{proof}

\begin{prop}\label{pro piXY}
We have $\tilde{\pi}^*\{ X, Y\} = \pm (64/3) \cdot \Eis^1(0,2)$.
\end{prop}

\begin{proof}
By Theorem \ref{thm H2E}(1) and Lemma \ref{lem XY}, the element $\tilde{\pi}^* \{X,Y\}$ lies in the $\varepsilon$-eigenspace. Thus
\begin{equation*}
\tilde{\pi}^* \{X,Y\} = \Pi_\varepsilon(\tilde{\pi}^* \{X,Y\}) = \Pi_\varepsilon \circ \alpha ((\iota^* \tilde{\pi}^* X) \otimes (\tilde{\pi}^* Y)) = \Pi_\varepsilon \circ \alpha (\tilde{\pi}^* Y \otimes \tilde{\pi}^* Y).
\end{equation*}
The divisor of $\tilde{\pi}^* Y$ is given by Lemma \ref{div piY}, and an explicit computation shows that
\begin{equation*}
\mu(\dv(\tilde{\pi}^* Y) \otimes \dv(\tilde{\pi}^* Y)) = - 4 (0,2) + 4 (0,6).
\end{equation*}
Using diagram (\ref{diag alpha eps}), we see that $\tilde{\pi}^* \{X,Y\} = (1/64) \cdot\mathcal{E}^1_\mathcal{M}(-4(0,2)+4(0,6))$. Using Corollary \ref{cor Eis1}, we get $\tilde{\pi}^* \{X,Y\} = \pm (32/3) \cdot (\Eis^1(0,2)-\Eis^1(0,6))$. Since $\Eis^1(0,6)=-\Eis^1(0,2)$, the proposition follows.
\end{proof}

\section{Deligne--Beilinson cohomology}\label{section:Deligne cohomology}

In this section, we briefly recall some facts on Deligne--Beilinson cohomology following Deninger \cite[\S 1]{Deninger1997}. This provides a convenient setting for the various differential forms that appear in the next sections.

Let $X$ be a smooth variety over $\C$. For any $i \geq 1$, the Deligne--Beilinson cohomology group $H^i_{\mathcal{D}}(X,\R(i))$ is defined by
\[
H^i_{\mathcal{D}}(X, \R(i)) = \{
\phi\in\mathcal{A}^{i-1}(X,\R(i-1))|d\phi = \pi_{i-1}(\omega),\omega\in F^i(X)\}
/d\mathcal{A}^{i-2}(X,\R(i-1)),
\]
where $\mathcal{A}^i(X,\R(j))$ are the smooth $i$-forms on $X$ with values in $\R(j) = (2\pi i)^j\R$, $F^i(X)$ is the space of holomorphic $i$-forms on $X$ with at most logarithmic singularities at infinity, and $\pi_n(\alpha)=\frac{1}{2}(\alpha+(-1)^n\overline{\alpha})$.

Note that the map $\phi \mapsto \omega$ induces a well-defined linear map $H^i_{\mathcal{D}}(X, \R(i)) \to F^i(X)$.

As an example, for any invertible function $f$ in $\mathcal{O}(X)^\times$, the function $\log\abs{f}$ defines an element of $H^1_{\mathcal{D}}(X,\R(1))$, and the associated holomorphic $1$-form is $\operatorname{dlog} f$.

There is a cup product on the Deligne--Beilinson cohomology groups
\[
\cup: H^i_{\mathcal{D}}(X,\R(i))\times H^j_{\mathcal{D}}(X,\R(j))\rightarrow H^{i+j}_{\mathcal{D}}(X,\R(i+j))
\]
that is defined as follows. For $\alpha=i,j$  consider a class $[\phi_\alpha]\in  H^\alpha_{\mathcal{D}}(X,\R(\alpha))$ represented by $\phi_\alpha\in\mathcal{A}^{\alpha-1}(X,\R(\alpha-1))$ with $d(\phi_\alpha)=\pi_{\alpha-1}(\omega_\alpha)$ for $\omega_\alpha\in F^\alpha(X)$. Then the cup product
\[
\phi_i\cup\phi_j= \phi_i\wedge\pi_j\omega_j+(-1)^i\pi_i\omega_i\wedge\phi_j
\] 
is in $\mathcal{A}^{i+j-1}(X,\R(i+j-1))$ and satisfies
\[
d(\phi_i\cup\phi_j) = \pi_{i+j-1}(\omega_i\wedge\omega_j).
\]
The cup product of the classes $[\phi_i]$ and $[\phi_j]$ is given by
\[
[\phi_i]\cup [\phi_j]=[\phi_i\cup\phi_j]\in H^{i+j}_{\mathcal{D}}(X,\R(i+j)).
\]
Let us now turn to the Beilinson regulator. Beilinson defined a $\Q$-linear map
\[
\operatorname{reg} : H^i_{\mathcal{M}}(X,\Q(i)) \to H^i_{\mathcal{D}}(X,\R(i))
\]
where $H^i_\mathcal{M}(\cdot,\Q(i))$ denotes motivic cohomology. In the simplest case $i=1$, we have $H^1_{\mathcal{M}}(X,\Q(1)) \cong \mathcal{O}(X)^\times \otimes \Q$, and the image of an invertible function $f \in \mathcal{O}(X)^\times$ under the Beilinson regulator is simply the class of $\log |f|$. Since the Beilinson regulator map is compatible with taking cup products, this determines the regulator of an arbitrary Milnor symbol $\{f_1,\ldots,f_i\} = f_1 \cup \ldots \cup f_i$ in $H^i_\mathcal{M}(X,\Q(i))$.

\section{Integrating Eisenstein symbols over Shokurov cycles}
\label{section:Integrating Eisenstein symbols}
In this section we study integrals of Eisenstein symbols over Shokurov cycles. Since Eisenstein symbols have logarithmic singularities at the cusps, the integral over a Shokurov cycle does not always converge. One way to circumvent this problem is to regularise the integrals as in \cite{brunault:regRZ}. Here we give a criterion for when the integrals are absolutely convergent. This criterion will always be satisfied in our applications.

We first discuss realisations of Eisenstein symbols. Consider the Eisenstein symbol $\Eis^n(u)$, where $n \geq 0$ and $u \in (\Z/N\Z)^2$. In the case $n \geq 1$, Deninger defined in \cite[\S 3]{deninger:extensions} a canonical representative $\Eis^n_{\mathcal{D}}(u)$ of the Beilinson regulator of $\Eis^n(u)$. It is a $\R(n)$-valued $n$-form on $E(N)^n(\C)$ satisfying $d\Eis^n_{\mathcal{D}}(u) = \pi_n(\Eis^n_{\mathrm{hol}}(u))$ where $\Eis^n_{\mathrm{hol}}(u)$ is a holomorphic Eisenstein series of weight $n+2$ for the group $\Gamma(N)$. Note that $\Eis^n_{\mathcal{D}}(u)$ is a real analytic Eisenstein series \cite[Prop. 5.4]{brunault:regRZ}.

We extend the definition of $\Eis^n_{\mathcal{D}}(u)$ and $\Eis^n_{\mathrm{hol}}(u)$ to the case $n=0$ by setting $\Eis^0_{\mathcal{D}}(u) = (2/N) \log |g_u|$ and $\Eis^0_{\mathrm{hol}}(u) = (2/N) \operatorname{dlog}(g_u)$ for every $u \neq (0,0)$. We can check that these definitions are consistent with \cite{brunault:regRZ} by applying the logarithm to the $q$-product expression for Siegel units. In particular, this argument shows that the Fourier expansions of $\Eis^0_{\mathcal{D}}(u)$ and $\Eis^0_{\mathrm{hol}}(u)$ are indeed given by the formulas in \cite[\S 8]{brunault:regRZ}. (Note that in the case $n=0$, the series appearing in \cite[Prop. 5.4]{brunault:regRZ} does not converge absolutely, and making sense of it requires Hecke summation.)

We now give the transformation property of Eisenstein symbols with respect to the group $\GL_2(\Z/N\Z)$.

\begin{lem}\label{g EisnD}
For any $g \in \GL_2(\Z/N\Z)$, we have
\begin{align*}
g^* \Eis^n_{\mathcal{D}}(u) & = \Eis^n_{\mathcal{D}}(ug),\\
g^* \Eis^n_{\mathrm{hol}}(u) & = \Eis^n_{\mathrm{hol}}(ug).
\end{align*}
\end{lem}

\begin{proof}
In the case $n \geq 1$, this is a direct consequence of Deninger's definition, while in the case $n=0$ this follows from the transformation property of Siegel units given in \S \ref{siegel units}.
\end{proof}

\begin{dfn}
Let $p:E(N) \to Y(N)$ be the projection map. Let $u, v\in (\Z/N\Z)^2$, with $u \neq (0,0)$. We define the Deninger--Scholl element $\Eis^{0,1}(u,v) := (p^* \Eis^0(u)) \cup \Eis^1(v)\in H^3_{\mathcal{M}}(E(N),\Q(3))$. 
\end{dfn}
In \cite[\S 6]{brunault:regRZ}, we defined a canonical differential $2$-form $\Eis^{0,1}_{\mathcal{D}}(u,v)$ on $E(N)(\C)$ representing the regulator of $\Eis^{0,1}(u,v)$. We split $\Eis^{0,1}_{\mathcal{D}}(u,v)$ into a sum $\eta^1_{u,v} + \eta^2_{u,v}$, where
\begin{align}
\label{def eta1} \eta^1_{u,v} & = p^* \Eis^0_\mathcal{D}(u) \cdot \pi_2(\Eis^1_{\mathrm{hol}}(v)), \\
\label{def eta2} \eta^2_{u,v} & = -p^* \pi_1(\Eis^0_{\mathrm{hol}}(u)) \wedge \Eis^1_\mathcal{D}(v).
\end{align}
It is convenient to extend the definition of Eisenstein symbols linearly as follows.
 For $\phi \in \C[(\Z/N\Z)^2 \times (\Z/N\Z)^2]$ such that $\phi((0,0),v) = 0$ for all $v\in (\Z/N\Z)^2$, we define
\begin{equation*}
\Eis^{0,1}_{\mathcal{D}}(\phi) = \sum_{u,v \in (\Z/N\Z)^2} \phi(u,v) \Eis^{0,1}_{\mathcal{D}}(u,v).
\end{equation*}
We define $\eta^1_\phi$ and $\eta^2_\phi$ analogously.

The following lemma is an immediate consequence of Lemma \ref{g EisnD}.

\begin{lem}\label{g Eis01D}
For any $g \in \GL_2(\Z/N\Z)$
\begin{align*}
g^* \Eis^{0,1}_{\mathcal{D}}(u,v) & = \Eis^{0,1}_{\mathcal{D}}(ug,vg),\\
g^* \Eis^{0,1}_{\mathrm{hol}}(u,v) & = \Eis^{0,1}_{\mathrm{hol}}(ug,vg).
\end{align*}
\end{lem}
Introducing a right $\GL_2(\Z/N\Z)$-action on $\C[(\Z/N\Z)^2 \times (\Z/N\Z)^2]$ by setting $\phi|g(u,v) = \phi(ug^{-1},vg^{-1})$, Lemma \ref{g Eis01D} can be rephrased as $g^* \Eis^{0,1}_{\mathcal{D}}(\phi) = \Eis^{0,1}_{\mathcal{D}}(\phi|g)$.

\begin{prop}\label{prop:abs convergence}
\begin{enumerate}
\item\label{prop: abs convergence item 1} For all integers $m,n\in\Z$ the integral $\int_{(mX+nY)\{0,\infty\}}\eta^2_{u,v}$ converges absolutely.
\item\label{prop: abs convergence item 2} If $\phi \in \C[(\Z/N\Z)^2 \times (\Z/N\Z)^2]$ satisfies
\begin{equation}\label{eqn: abs convergence cond 1}
\sum_{u,v \in (\Z/N\Z)^2} \phi(u,v) B_2(\{u_1/N\})B_3(\{v_1/N\})=0
\end{equation}
and
\begin{equation}\label{eqn: abs convergence cond 2}
\sum_{\substack{u,v \in (\Z/N\Z)^2\\ u_1 = 0}} \phi(u,v) \log|1-\zeta_N^{u_2}|B_3(\{v_1/N\})=0,
\end{equation}
then
$
\int_{X\{0,\infty\}} \eta^1_\phi
$
converges absolutely.
\item\label{prop: abs convergence item 3} If $\phi|\sigma$ satisfies conditions \eqref{eqn: abs convergence cond 1} and \eqref{eqn: abs convergence cond 2}, then
$
\int_{Y\{0,\infty\}} \eta^1_\phi
$
converges absolutely.
\end{enumerate}
\end{prop}
\begin{proof}
We start with item \eqref{prop: abs convergence item 2}. The strategy is to first compute the integral of $\Eis^{0,1}_\mathcal{D}(u,v)$ along the fibres of the fibration $p_Q : Q \{0,\infty\} \to \{0,\infty\}$. Using \cite[Prop. 8.1]{brunault:regRZ}, we get
\begin{equation}\label{eqn:integral over fibre 1}
\int_{p_X} \eta^1_{u,v} = \Eis^0_\mathcal{D}(u) \cdot -\frac{6\pi^2}{N} \left(F^{(3)}_{-v}(\tau) \tau d\tau +  \overline{F^{(3)}_{-v}(\tau)} \overline{\tau} d\overline{\tau}\right),
\end{equation}
so we need to show that the integral
\begin{equation}\label{eqn:absolute convergence}
\int_0^\infty  \left|\sum_{u,v \in (\Z/N\Z)^2} \phi(u,v)\Eis^0_\mathcal{D}(u) \cdot \left(F^{(3)}_{-v}(iy)  +  \overline{F^{(3)}_{-v}(iy)}\right)\right|y dy
\end{equation}
converges. The Fourier expansion of $\Eis^0_\mathcal{D}(u)$ can be deduced from \cite[Lemma 16]{Brunault2016}:
\begin{equation*}
\Eis^0_\mathcal{D}(u) = \frac{2}{N} \log |g_u| = \frac{2}{N} \Re\left(\pi i B_2(\{\frac{a}{N}\}) \tau + \delta_0(a) \log (1-\zeta_N^b) + \sum_{n \geq 1} c_n q^{n/N}\right).
\end{equation*}
In particular $\Eis^0_\mathcal{D}(u)$ grows at most polynomially for $y\to\infty$. The Fourier expansion of $F^{(3)}_{-v}(\tau)$ is given in \cite[Lemme 3.3]{brunault:regRZ}. It has the constant term $-\frac{B_3(\{-v_1/N\})}{3}$, so the conditions on $\phi$ are equivalent to the assumption that the integrand in \eqref{eqn:absolute convergence} decays exponentially for $y\to\infty$. Using \textit{loc. cit.} we can also determine the Fourier expansion of $W_N F^{(3)}_{-v}(\tau) = -i N^{3/2} \tau^{-3}  F^{(3)}_{-v}(-1/(N\tau))$. It has constant term $-i N^{3/2}\zeta(-v_2/N, -2)$, which implies that $F^{(3)}_{-v}(iy)  +  \overline{F^{(3)}_{-v}(iy)}$ decays exponentially for $y\to 0$. The rest of the integrand can in fact be bounded for $y\to 0$, so this concludes the proof that \eqref{eqn:absolute convergence} is absolutely convergent. The third statement in the proposition follows from Lemmas \ref{action G shokurov} and \ref{g Eis01D}.

To show the absolute convergence of $\int_{Q\{0,\infty\}}\eta^2_{u,v}$ we calculate
\begin{equation}\label{eqn:integral over fibre 2}
\int_{p_Q} \eta^2_{u,v} =- \pi_1(\Eis^0_{\mathrm{hol}}(u))
 \frac{3}{2\pi i N}iy\left(F_{\sigma v}^{0,1}(\tau)(m\overline{\tau}+n) + F_{\sigma v}^{1,0}(\tau)(m\tau+n)
\right)
\end{equation}
where
\[
\pi_1\Eis^0_{\text{hol}}(u) = -\frac{2i\pi}{N} (F_{-u}^{(2)}(\tau)d\tau + \overline{F_{-u}^{(2)}(\tau)}d\overline{\tau}).
\]
Restricted to $\{0,\infty\}$ the differential form $\pi_1\Eis^0_{\text{hol}}(u)$ becomes $\frac{2\pi}{N}(F_{-u}^{(2)}(iy) - \overline{F_{-u}^{(2)}(iy)})dy $ and similar calculations as before show that $F_{-u}^{(2)}(iy) - \overline{F_{-u}^{(2)}(iy)}$ decays exponentially for $y\to\infty$ and $y\to 0$.
\end{proof}

Even if absolute convergence is given, some further care is needed when integrating Eisenstein symbols: contrary to what \cite[Remark 2]{Brunault2016} and \cite[Remarque 1.2(iii)]{brunault:regRZ} suggest, it is not true in general that if $\omega$ is a closed differential $2$-form on $E(N)(\C)$ and $\alpha,\beta,\gamma$ are cusps, then $\int_{P\{\alpha,\beta\}} \omega + \int_{P\{\beta,\gamma\}} \omega = \int_{P\{\alpha,\gamma\}} \omega$. The reason is that $\omega$ may have non trivial residues at the cusps. This motivates the following definition.

\begin{dfn}
Let $n \geq 1$ be an integer, and let $\omega$ be a closed differential $n$-form on $E(N)^{n-1}(\C)$. We say that \emph{$\omega$ has trivial residues at the cusps} if for every $g \in \GL_2(\Z/N\Z)$, every $Q \in \Z[X,Y]_{n-1}$ and every $\alpha \in \R$, we have
\begin{equation*}
\int_{Q \{iy,\alpha+iy\}} g^* \omega \xrightarrow[y \to \infty]{} 0.
\end{equation*}
\end{dfn}

\begin{exmp}
In the case $n=1$, let $u,v$ be modular units on $Y(N)(\C)$ such that the Milnor symbol $\{u,v\}$ has trivial tame symbols at all cusps. Then the $1$-form $\eta(u,v) = \log |u| \mathrm{darg}(v) - \log |v| \mathrm{darg}(u)$ has trivial residues at the cusps.
\end{exmp}

\begin{rem}
In general, let us consider the differential form $ \omega = \Eis^{k_1,k_2}_{\mathcal{D}}(u,v)$ from \cite[\S 6]{brunault:regRZ}. The property of $\omega$ having trivial residues at the cusps is probably related to the motivic element $\Eis^{k_1,k_2}(u,v)$ extending to Deligne's smooth compactification of $E(N)^{k_1+k_2}$, but we will not need this property in this article.
\end{rem}

\begin{prop}\label{Res Eis01}
Let $u=(a,b) \in (\Z/N\Z)^2$, $u \neq (0,0)$, and $v=(c,d) \in (\Z/N\Z)^2$. We have
\begin{align}
\label{Res Eis01 first}
\lim_{y \to \infty} \int_{X\{iy,\alpha+iy\}} \Eis^{0,1}_\mathcal{D}(u,v) = & -\frac{4\pi^2}{N^2} \delta_0(a) \log |1-\zeta_N^b| B_3(\{c/N\}) \alpha^2 \\
\nonumber  & + \frac{3\pi i}{N^2} B_2(\{a/N\}) \delta_0(c) \left(\hat{\zeta}(-d/N,2)-\hat{\zeta}(d/N,2)\right) \alpha \\
\label{Res Eis01 second} \lim_{y \to \infty} \int_{Y\{iy,\alpha+iy\}} \Eis^{0,1}_\mathcal{D}(u,v) = & -\frac{8\pi^2}{N^2} \delta_0(a) \log |1-\zeta_N^b| B_3(\{c/N\}) \alpha.
\end{align}
\end{prop}

\begin{proof}
Again we first compute the integral of $\Eis^{0,1}_\mathcal{D}(u,v)$ along the fibres of the fibration $p_{Q,\alpha} : Q \{iy,\alpha+iy\} \to \{iy,\alpha+iy\}$. We start from equation \eqref{eqn:integral over fibre 1} with $p_X$ replaced by $p_{X,\alpha}$ and integrate from $\tau=iy$ to $\tau=\alpha+iy$. Since every term involving $q^{n/N}=\exp(2\pi in\tau/N)$ with $n \geq 1$ will vanish when taking the limit as $y \to \infty$ this gives
\begin{equation}\label{int eta1}
\int_{X\{iy,\alpha+iy\}} \eta^1_{u,v} = -\frac{8\pi^2}{N^2} B_3(\{\frac{c}{N}\}) \left(-\pi B_2(\{\frac{a}{N}\}) y + \delta_0(a) \log |1-\zeta_N^b| \right) \frac{\alpha^2}{2} + o_{y \to \infty}(1).
\end{equation}
The integral of $\eta^2_{u,v}$ can be computed similarly, using the Fourier expansions \cite[Lemme 3.3, Prop. 8.1]{brunault:regRZ} and \cite[(20), Prop. 8.4]{brunault:regRZ}. We get
\begin{align}\label{int eta2}
\int_{X\{iy,\alpha+iy\}} \eta^2_{u,v} = & \frac{3}{2N^2} B_2(\{\frac{a}{N}\}) \left(-\frac{8\pi^3}{3} B_3(\{\frac{c}{N}\}) y\alpha^2\right. \\
\nonumber & \left.+2\pi i \delta_0(c) \left(\hat{\zeta}(-\frac{d}{N},2)-\hat{\zeta}(\frac{d}{N},2)\right) \alpha\right) + o_{y \to \infty}(1).
\end{align}
Adding (\ref{int eta1}) and (\ref{int eta2}), the terms in $y$ cancel out, giving (\ref{Res Eis01 first}). The identity (\ref{Res Eis01 second}) follows from a similar computation.
\end{proof}
\begin{cor}\label{cor trivial residues}
Let $b,b',d \in \Z/N\Z$ such that $d$ is a multiple of both $b$ and $b'$. Suppose furthermore that $b\equiv \pm b'\mod\gcd(d,N)$. Then the form 
\[
\Eis^{0,1}_{\mathcal{D}}((0,b),(0,d))-  \Eis^{0,1}_{\mathcal{D}}((0,b'),(0,d))
\]
has trivial residues at the cusps.
\end{cor}
\begin{proof}
Let $\left(\begin{smallmatrix} x&y\\w&z\end{smallmatrix}\right)\in \GL_2(\Z/N\Z)$. We need to show that the two limits calculated in Proposition \ref{Res Eis01} vanish for the form $\Eis^{0,1}_{\mathcal{D}}((bw,bz),(dw,dz))-  \Eis^{0,1}_{\mathcal{D}}((b'w,b'z),(dw,dz))$. If either $bw$ or $b'w$ is zero, then $dw$ is zero since $d$ is a multiple of $b$ and $b'$. This implies that the limit \eqref{Res Eis01 second} vanishes and only the second term of \eqref{Res Eis01 first} remains. If this term is non-zero, then we must have $dw=0$, so $\frac{N}{\gcd(d,N)}|w$. By our assumptions this implies $bw \equiv \pm b'w \mod N$ and so $B_2(\{bw/N\}) = B_2(\{b'w/N\})$ and the contributions of $\Eis^{0,1}_{\mathcal{D}}((bw,bz),(dw,dz))$ and $\Eis^{0,1}_{\mathcal{D}}((b'w,b'z),(dw,dz))$ to \eqref{Res Eis01 first} cancel each other.
\end{proof}

\section{Calculation}
\label{section:Calculation}

We first express the Mahler measure of $P=X+\frac{1}{X}+Y+\frac{1}{Y}+Z+\frac{1}{Z}-2$ as the integral of a closed differential form on $E_1(8)(\C)$.

By Lemma \ref{lem:lecacheux param}, we may and will identify the Deninger cycle $D'_P = D_P \backslash \{(-1,-1,3+2\sqrt{2})\}$ with a cycle on $E_1(8)(\C)$.

Note that for every $b \in \Z/N\Z$, $b \neq 0$, and every $d \in \Z/N\Z$, the $2$-form $\Eis^{0,1}_{\mathcal{D}}((0,b),(0,d))$ descends to a form on $E_1(N)(\C)$.

\begin{lem}\label{mP eta1}
We have
\begin{equation*}
m(P) = \pm \frac{1}{(2\pi i)^2} \int_{D'_P} \eta_1
\end{equation*}
where $\eta_1$ is the closed $2$-form on $E_1(8)(\C)$ defined by
\begin{equation*}
\eta_1 = \frac{2^9}{3} \left(\Eis^{0,1}_{\mathcal{D}}((0,3),(0,2))-\Eis^{0,1}_{\mathcal{D}}((0,1),(0,2)) \right).
\end{equation*}
\end{lem}

\begin{proof}
Since $P^*=1$, Jensen's formula (\ref{jensen}) gives
\begin{equation*}
m(P)=\frac{1}{(2\pi i)^2} \int_{D'_P} \log \abs{Z} \frac{dX}{X} \wedge \frac{dY}{Y}.
\end{equation*}
Since $m(P)$ is real, we may write
\begin{equation*}
m(P)=\frac{1}{(2\pi i)^2} \int_{D'_P} \log \abs{Z} \pi_2\left(\frac{dX}{X} \wedge \frac{dY}{Y}\right).
\end{equation*}
Let $\Omega^2(E_1(8)(\C))$ be the space of holomorphic $2$-forms on $E_1(8)(\C)$. There is a well-defined map
\begin{equation*}
H^2_{\mathcal{D}}(E_1(8)(\C),\R(2)) \to \Omega^2(E_1(8)(\C))
\end{equation*}
sending the class of a $1$-form $\phi$ to the unique holomorphic form $\omega$ such that $d\phi = \pi_1(\omega)$. By Proposition \ref{pro piXY}, we have $\{X,Y\}=\pm (64/3) \Eis^1(0,2)$ in $H^2_{\mathcal{M}}(E_1(8),\Q(2))$. Taking regulators and applying the above map, we see that
\begin{equation*}
\frac{dX}{X} \wedge \frac{dY}{Y} = \pm \frac{64}{3} \Eis^1_{\mathrm{hol}}(0,2).
\end{equation*}
Since $Z$ takes positive real values on $D$, the form $\pi_1(dZ/Z) = i \operatorname{darg}(Z)$ vanishes on $D'_P$. We thus have
\begin{equation*}
m(P)=\frac{1}{(2\pi i)^2} \int_{D'_P} \log \abs{Z} \pi_2\left(\pm \frac{64}{3} \Eis^1_{\mathrm{hol}}(0,2)\right) - \pi_1\left(\frac{dZ}{Z}\right) \wedge \pm \frac{64}{3} \Eis^1_{\mathcal{D}}(0,2).
\end{equation*}
It remains to express $\log |Z|$ and $dZ/Z$ in terms of Eisenstein symbols using Propositions \ref{proZ1} and \ref{proZ2}. Letting $p_1 : E_1(8) \to Y_1(8)$, we have
\begin{align*}
\log \abs{Z} & = p_1^* \left( 8 (\Eis^0_{\mathcal{D}}(0,3)-\Eis^0_{\mathcal{D}}(0,1)) \right) \\
\frac{dZ}{Z} & = p_1^* \left(8 (\Eis^0_{\mathrm{hol}}(0,3)-\Eis^0_{\mathrm{hol}}(0,1)) \right)
\end{align*}
The result now follows from the definition of $\Eis^{0,1}_{\mathcal{D}}$ (see (\ref{def eta1}) and (\ref{def eta2})).
\end{proof}

We now express the Mahler measure of $P$ as an integral over a suitable Shokurov cycle.

\begin{lem}\label{diffeo Ztau}
The map $\tau \mapsto Z(\tau)$ induces a diffeomorphism $\{1/2,\infty\} \xrightarrow{\cong} (1,3+2\sqrt{2})$.
\end{lem}

\begin{proof}
By Proposition \ref{proZ1}, we have $Z(1/2+it) \in \R$ for every $t>0$. By Propositions \ref{prop:siegel-transformation} and \ref{prop:evaluate siegel unit} we can evaluate $Z$ at any cusp. We find $Z(i\infty) = 3+2\sqrt{2}$ and $Z(1/2) = \lim_{t\to\infty}(g_{6,3}(it) / g_{2,1}(it))^2 = 1$. A more detailed analysis will show that $Z$ is monotonly increasing on the path $t \mapsto 1/2+it$.

Note that $Z$ is an Hauptmodul for $\Gamma_1(8)$ and is defined over $\R$, so that $Z$ identifies the real points $Y_1(8)(\R)$ with an open subset of $\R$. It therefore suffices to show that the canonical map $\{1/2,\infty\} \to Y_1(8)(\R)$ is injective. This follows from the work of Snowden on real components of modular curves \cite{snowden:real}. We use \cite[Lemma 3.3.4]{snowden:real} for the modular curve $\Gamma_1(8) \backslash \h$ endowed with the complex conjugation $c(\tau)=-\overline{\tau}$. The matrix $\gamma = \left(\begin{smallmatrix} 1 & -1 \\ 0 & 1 \end{smallmatrix}\right) \in \Gamma_1(8)$ is admissible and $C^{o}_\gamma = \{\tau \in \h : c(\tau)=\gamma \tau\} = \{1/2,\infty\}$. Since $\Gamma_1(8)$ acts freely on $\h$, there are no elliptic points on $\Gamma_1(8) \backslash \h$ and the only element of $\Gamma_1(8)$ leaving $C^o_\gamma$ invariant is the identity. We may thus take $\mathcal{F}=\{1/2,\infty\}$ in \textit{loc.}\ \textit{cit.}, giving the desired result.
\end{proof}

By Lemmas \ref{fibration DP} and \ref{diffeo Ztau}, the Deninger cycle $D'_P$ is endowed with a fibration $q : D'_P \to \{1/2,\infty\}$. Note that for every $\tau \in \{1/2,\infty\}$, the fibre $q^{-1}(\tau)$ can be identified with a closed $1$-cycle on the elliptic curve $E_\tau \cong \C/(\Z+\tau \Z)$. Since $\tau \in \{1/2,\infty\}$, the elliptic curve $E_\tau$ is defined over $\R$.

\begin{lem}\label{lem Dtau}
For every $\tau \in \{1/2,\infty\}$, the class of $q^{-1}(\tau)$ generates $H_1(E_\tau,\Z)^+$.
\end{lem}

\begin{proof}
The complex conjugation $c(X,Y,Z)=(\overline{X},\overline{Y},\overline{Z})$ on $E_1(8)(\C)$ leaves $D'_P$ stable and preserves its orientation. Since $c$ fixes pointwise the cycle $\{1/2,\infty\}$ on $Y_1(8)(\C)$, we deduce that $c$ leaves $q^{-1}(\tau)$ stable and preserves its orientation. Thus the class of $q^{-1}(\tau)$ belongs to $H_1(E_\tau,\Z)^+$.

Let $\gamma_\tau^+ : [0,1] \to E_\tau$ denote the canonical generator of $H_1(E_\tau,\Z)^+$, defined by $\gamma_\tau^+(t) = [t]$. Since the fibre $q^{-1}(\tau)$ varies continuously with $\tau$, we must have $[q^{-1}(\tau)] = m [\gamma_\tau^+]$ for some integer $m$ not depending on $\tau$.

In order to determine $m$, we introduce for every $\tau \in \{1/2,\infty\}$ the holomorphic form $\omega_\tau$ on $E_\tau$ defined by
\begin{equation*}
\omega_\tau = \frac{dU}{2V+(Z(\tau)+\frac{1}{Z(\tau)}-2)U}
\end{equation*}
Using the function \texttt{ellpointtoz} of Pari/GP or integrating $\omega_\tau$ numerically, we check that for $\tau_0=1/2+i$, we have
\begin{equation*}
\int_{q^{-1}(\tau_0)} \omega_{\tau_0} \approx \int_{\gamma_{\tau_0}^+} \omega_{\tau_0}.
\end{equation*}
Since we know a priori that $m$ is an integer, we deduce that $m=1$.
\end{proof}

In view of Lemma \ref{lem Dtau}, we are naturally led to search for a Shokurov cycle on $E_1(8)(\C)$ sharing the same properties with $D'_P$. By definition, the Shokurov cycle $Y\{1/2,\infty\}_1$ does the job: it is endowed with a fibration over $\{1/2,\infty\}$ and its fibre over $\tau$ is equal to $\gamma_\tau^+$.

\begin{prop}
We have
\begin{equation*}
\int_{D'_P} \eta_1 = \int_{Y\{1/2,\infty\}_1} \eta_1.
\end{equation*}
\end{prop}

\begin{proof}
By integrating over the fibres of $q : D'_P \to \{1/2,\infty\}$, we get
\begin{equation*}
\int_{D'_P} \eta_1 = \int_{1/2}^\infty \int_{q} \eta_1.
\end{equation*}
Similarly, using the fibration $p_Y : Y\{1/2,\infty\}_1 \to \{1/2,\infty\}$ we have
\begin{equation*}
\int_{Y\{1/2,\infty\}_1}  \eta_1 = \int_{1/2}^\infty \int_{p_Y} \eta_1.
\end{equation*}
By Lemma \ref{lem Dtau}, the fibres $q^{-1}(\tau)$ and $p_Y^{-1}(\tau)$ are homologous, hence $\int_q \eta_1 = \int_{p_Y} \eta_1$.
\end{proof}

Pulling back $\eta_1$ to $E(8)(\C)$, we define
\begin{equation*}
\eta = \tilde{\pi}^* \eta_1 = \frac{2^9}{3} \left(\Eis^{0,1}_{\mathcal{D}}((0,3),(0,2))-\Eis^{0,1}_{\mathcal{D}}((0,1),(0,2)) \right).
\end{equation*}
Using Lemma \ref{shokurov E E1}, we get
\begin{equation}\label{calculation eq1}
m(P) = \pm \frac{1}{(2\pi i)^2} \int_{Y\{1/2,\infty\}} \eta.
\end{equation}

In order to proceed further, we decompose the modular symbol $Y\{1/2,\infty\}$ as a sum of Manin symbols of the form $g_* (X\{0,\infty\})$ with $g \in \GL_2(\Z/8\Z)$. We have
\begin{align}
\nonumber Y \left\{\frac12,\infty\right\} & = Y \left\{\frac12,0 \right\} + Y \{0,\infty\} \\
\nonumber & = \begin{pmatrix} 0 & -1 \\ 1 & -2 \end{pmatrix} \cdot (X-2Y) \{0,\infty\} - \begin{pmatrix} 0 & -1 \\ 1 & 0 \end{pmatrix} \cdot X \{0,\infty\} \\
\label{sum gX0oo} & = \begin{pmatrix} 0 & -1 \\ 1 & -2 \end{pmatrix} \cdot X \{0,\infty\} + 2 \begin{pmatrix} -1 & 0 \\ -2 & -1 \end{pmatrix} \cdot X \{0,\infty\}  - \begin{pmatrix} 0 & -1 \\ 1 & 0 \end{pmatrix} \cdot X \{0,\infty\}.
\end{align}

At this point we face two problems. Firstly, the integrals of $\eta$ over the Shokurov cycles $Y\{1/2,0\}$ and $Y\{0,\infty\}$ do not converge absolutely. Secondly, we must prove that the integral of $\eta$ over a sum of modular symbols is equal to the sum of the integrals of $\eta$ over these modular symbols (whenever these integrals are convergent).

In order to overcome these issues, we consider the hyperbolic ideal triangle with vertices $0$, $1/2$ and $\infty$, and we truncate it by cutting along horocycles centered at these cusps. We get the hexagon shown below.

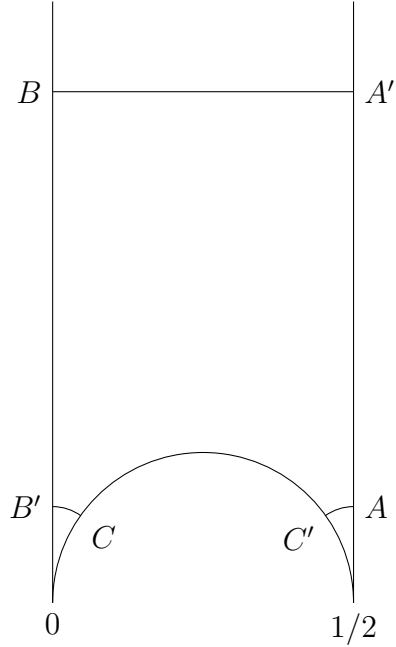
\begin{equation*}
\begin{tikzpicture}[scale=8]
\draw [name path=BB'] (0,1)--(0,0) node [below] {$0$};
\draw [name path=CC'] (0,0) arc (180:0:0.25) node [below] {$1/2$};
\draw [name path=AA'] (0.5,0) -- (0.5,1);
\draw (0,0.85) node [left] {$B$} -- (0.5,0.85) node [right] {$A'$};
\draw (0,0.16) [name path=B'C] arc (90:54:0.08);
\draw (0.5,0.16) [name path=C'A] arc (90:126:0.08);
\draw [name intersections={of=BB' and B'C}] 
(intersection-1) node [left] {$B'$};
\draw [name intersections={of=B'C and CC'}] 
(intersection-1) node [below right] {$C$};
\draw [name intersections={of=CC' and C'A}] 
(intersection-1) node [below left] {$C'$};
\draw [name intersections={of=C'A and AA'}] 
(intersection-1) node [right] {$A$};
\end{tikzpicture}
\end{equation*}

Since $\eta$ is closed, Stokes' theorem on the domain $AA'BB'CC'$ implies
\begin{equation*}
\left(\int_{Y\{A,A'\}} + \int_{Y\{A',B\}}  + \int_{Y\{B,B'\}} + \int_{Y\{B',C\}} + \int_{Y\{C,C'\}}  + \int_{Y\{C',A\}} \right) \eta = 0.
\end{equation*}
By Corollary \ref{cor trivial residues}, the $2$-form $\eta$ has trivial residues at the cusps. It follows that when $A,A',B,B',C,C'$ approach the cusps, we have
\begin{equation*}
\lim \int_{Y\{A',B\}} \eta = \lim \int_{Y \{B',C\}} \eta = \lim \int_{Y \{C',A\}} \eta = 0,
\end{equation*}
so
\begin{equation*}
\int_{Y\{1/2,\infty\}} \eta = \lim \left(\int_{Y\{C',C\}} \eta + \int_{Y \{B',B\}} \eta \right).
\end{equation*}
We now apply suitable matrices of $\SL_2(\Z)$ in order to reduce to integrals over the domain $X\{0,\infty\}$ as in (\ref{sum gX0oo}). Using Lemma \ref{action G shokurov} and Lemma \ref{lem int sum shokurov}, we get
\begin{equation}\label{calculation eq2}
\int_{Y\{1/2,\infty\}} \eta = \lim_{y \to \infty} \int_{X \{i/y,iy\}} \eta'
\end{equation}
with
\begin{equation*}
\eta' = \begin{pmatrix} 0 & -1 \\ 1 & -2 \end{pmatrix}^* \eta + 2 \begin{pmatrix} -1 & 0 \\ -2 & -1 \end{pmatrix}^* \eta - \begin{pmatrix} 0 & -1 \\ 1 & 0 \end{pmatrix}^* \eta.
\end{equation*}
Note that we used the fact that
\begin{equation*}
\int_{Y\{i/y,iy\}} \omega = - \int_{\sigma_* (X\{i/y,iy\})} \omega =  - \int_{X\{i/y,iy\}} \sigma^* \omega
\end{equation*}
for a closed $2$-form $\omega$ on $E(8)(\C)$.

Using Corollary \ref{g Eis01D}, we have
\begin{equation*}
\eta' = \frac{2^9}{3} \Eis^{0,1}_{\mathcal{D}}(\phi)
\end{equation*}
with
\begin{equation*}
\phi = [(3,2),(2,4)]-[(1,6),(2,4)]+2[(2,5),(4,6)]-2[(6,7),(4,6)]-[(3,0),(2,0)]+[(1,0),(2,0)].
\end{equation*}
By Proposition \ref{prop:abs convergence} the integral of $\Eis^{0,1}_{\mathcal{D}}(\phi)$ over $X\{0,\infty\}$ converges absolutely.

We finally may apply the main result of \cite{brunault:regRZ}. More precisely, we use \cite[Thm 1.1]{brunault:regRZ} for the terms $[(a,b),(c,d)]$ with $d \neq 0$, and we use \cite[Thm 9.5]{brunault:regRZ} for the term $-[(3,0),(2,0)]+[(1,0),(2,0)]$. We restate the two theorems in our situation. Recall the Eisenstein series
\[
G_{a,b}^{(k)}=a_0(G_{a,b}^{(k)})+\sum_{\substack{m,n\geq 1\\ m\equiv a, n\equiv b(N)}}m^{k-1}q^{mn}+(-1)^k\sum_{\substack{m,n\geq 1\\ m\equiv -a, n\equiv -b(N)}}m^{k-1}q^{mn}.
\]
The constant term $a_0(G_{a,b}^{(k)})$ is given in \cite[Def. 3.5]{brunault:regRZ}. They are modular forms of weight $k$ with respect to $\Gamma_1(N^2)$, except in the case $k=2$ and $a=0$. However in this case, for any $b'\in\Z/N\Z$, the function $G_{0,b}^{(2)}-G_{0,b'}^{(2)}$ is a modular form with respect to $\Gamma_1(N^2)$. 
\begin{thm}[Theorem 1.1 of \cite{brunault:regRZ}\label{thm:brunault-regrz-thm1.1} for $k_1=0$ and $k_2=1$]
Let $N\geq 3$ and $u=(a,b),v=(c,d)\in(\Z/N\Z)^2$ with $u\neq (0,0)$. If $d\neq 0$, then
\begin{equation}\label{thm regRZ}
\int_{X\{0,\infty\}}^* \Eis_{\mathcal{D}}^{0,1}(u,v) \\
= \frac{3}{N^{3}}(2\pi)^2
\Lambda^*(G^{(2)}_{d,a}G^{(1)}_{b,-c} - G^{(2)}_{d,-a}G^{(1)}_{b,c},0). 
\end{equation}
where $\Lambda^*(\cdot,0)$ denotes the regularized value at $s=0$ \cite[Def. 3.13]{brunault:regRZ}.
\end{thm}
In the case $d=0$, Theorem \ref{thm:brunault-regrz-thm1.1} also holds provided one replaces $u$ by a formal linear combination $\sum_i \lambda_i (a_i,b)$ satisfying $\sum_i \lambda_i=0$, taking the appropriate linear combination of Eisenstein series in the right hand side of (\ref{thm regRZ}), see \cite[Thm 9.5]{brunault:regRZ}.

%\begin{thm}[Theorem 9.5 of \cite{brunault:regRZ}\label{thm:brunault-regrz-thm9.5} for $k_1=0$ and $k_2=1$]
%Let $u=(a,b),u' = (a',b), v=(c,0)\in(\Z/N\Z)$ with $u,u'\neq (0,0)$. Then
%\begin{equation*}
%\int_{X\{0,\infty\}}^* \Eis_{\mathcal{D}}^{0,1}(u,v)-\Eis_{\mathcal{D}}^{0,1}(u',v) \\
%= \frac{3}{N^{3}}(2\pi)^2
%\Lambda^*((G^{(2)}_{0,a}-G^{(2)}_{0,a'})G^{(1)}_{b,-c} - (G^{(2)}_{0,-a}-G^{(2)}_{0,-a'})G^{(1)}_{b,c},0). 
%\end{equation*}
%\end{thm}

Using Theorem \ref{thm:brunault-regrz-thm1.1}, we get
\begin{equation}\label{calculation eq3}
\int_{X\{0,\infty\}} \eta' = \frac{2^9}{3} \int_{X\{0,\infty\}} \Eis^{0,1}_{\mathcal{D}}(\phi) = \frac{2^9}{3} \int_{X\{0,\infty\}}^* \Eis^{0,1}_{\mathcal{D}}(\phi) = \frac{2^9}{3} \cdot \frac{3\cdot (2\pi)^2}{8^3} \Lambda^*(F,0)
\end{equation}
with
\begin{align*}
F & = G^{(2)}_{4,3} G^{(1)}_{2,-2} - G^{(2)}_{4,-3}G^{(1)}_{2,2} - G^{(2)}_{4,1} G^{(1)}_{6,-2} + G^{(2)}_{4,-1} G^{(1)}_{6,2} + 2 G^{(2)}_{6,2} G^{(1)}_{5,-4} - 2 G^{(2)}_{6,-2} G^{(1)}_{5,4}\\
& \qquad - 2 G^{(2)}_{6,6} G^{(1)}_{7,-4} + 2 G^{(2)}_{6,-6} G^{(1)}_{7,4} - (G^{(2)}_{0,3}-G^{(2)}_{0,1}) G^{(1)}_{0,-2} + (G^{(2)}_{0,-3} - G^{(2)}_{0,-1}) G^{(1)}_{0,2}.
\end{align*}
We know in advance that $F(\tau/8)$ belongs to $M_3(\Gamma_1(8))$ by \cite[Remarque 1.2(iv)]{brunault:regRZ}, and a computation reveals that $F(\tau/8) = -4 f(\tau)$, where
\begin{equation*}
f(\tau) = q - 2q^2 - 2q^3 + 4q^4 + 4q^6 - 8q^8  - 5q^9 + 14q^{11}+O(q^{12})
\end{equation*}
is the unique newform of weight $3$ and level $\Gamma_1(8)$. We have
\begin{equation*}
\Lambda_{64}^*(F,0)=-4 \Lambda_8(f,0) = -4 \Lambda_8(W_8 f,3).
\end{equation*}

\begin{lem}
We have $W_8 f =  f$.
\end{lem}
\begin{proof}
Since $S_3(\Gamma_1(8))$ is $1$-dimensional, we have $W_8 f = \varepsilon f$ with $\varepsilon\in\{\pm 1\}$. Evaluating numerically $f$ and $W_8 f$ at $\tau=i/\sqrt{8}$, we get $\varepsilon=1$.
\end{proof}

It follows that
\begin{equation}\label{calculation eq4}
\Lambda^*(F,0) = -4 \Lambda(f,3).
\end{equation}
Putting (\ref{calculation eq1}), (\ref{calculation eq2}), (\ref{calculation eq3}) and (\ref{calculation eq4}) together, we obtain
\begin{equation*}
m(P) = \pm \frac{1}{(2\pi i)^2} (2\pi)^2 \cdot -4 \Lambda(f,3) = \pm 4\Lambda(f,3).
\end{equation*}
Since $m(P)$ and $\Lambda(f,3)$ are positive real numbers, we conclude that $m(P)=4\Lambda(f,3)$.

\section{Other examples}
\label{section:Other examples}

In this section we compute several more Mahler measures of models of elliptic surfaces in terms of $L$-values. Suppose we have an affine model of an elliptic modular surface $E = E(\Gamma)$ defined by a polynomial $P(X,Y,Z)\in\C[X,Y,Z]$ and an elliptic fibration given by $(X,Y,Z)\mapsto Z$. The method we developed relies on several properties of the fibration. 
\begin{enumerate}
\item \label{condition1} The horizontal divisors of $X$ and $Y$ are supported in the torsion sections of $E$ and the Milnor symbol $\{X,Y\}$, a priori defined on the complement of the torsion sections, extends to an element of $H^2_{\mathcal{M}}(E,\Q(2)) = K_2^{(2)}(E)$.
\item \label{condition2} The function $Z$ is the pull-back of a modular unit on the modular curve $X(\Gamma)$.
\item \label{condition3} The function $Z$ takes real values on the Deninger cycle $D_P=\{(X,Y,Z)\in V(P) : |X|=|Y|=1,~|Z|>1\}$ and $D_P$ is fibred over an interval $(a,b) \subset \R$ where $a$ and $b$ are cusps of $X(\Gamma)$.
\end{enumerate}

\begin{rem} More generally, it seems that our method could also work when the parametrisation $\phi : E(\Gamma) \dashrightarrow V(P)$ is a dominant map (not necessarily a birational isomorphism). In this case, the relevant condition would be that the Deninger cycle $D_P$ is homologous to the push-forward of a Shokurov cycle.
\end{rem}

In order to find good candidate polynomials $P$, we proceed as follows. We look for functions $X,Y$ on the elliptic surface $E$ whose divisors are supported in the torsion subgroup of $E$. Given $X$ and $Y$, we compute the tame symbols of $\{X,Y\}$ along the torsion sections and check whether there exist modular units $u$ and $v$ such that $\{uX,vY\}$ extends to an element of $H^2_{\mathcal{M}}(E,\Q(2)) = K_2^{(2)}(E)$. Once the functions $X,Y$ are found, we compute the minimal polynomial $P$ of $(X,Y,Z)$ and we try to adjust the modular unit $Z$ such that the condition (\ref{condition3}) is satisfied. Note however that in some cases like the universal elliptic curve for $\Gamma_1(7)$, the search for $X,Y$ was unsuccessful.

Given the three conditions above we can follow the method described in the previous sections and find that $m(P)-m(P^*)$ equals $\Lambda(F,0)$ for an explicit modular form $F\in \mathcal{M}_3(\Gamma)$ with rational Fourier coefficients. It is not always the case that $F$ is a cusp form and in the examples that follow we see that $F$ can be an Eisenstein series or a combination of an Eisenstein series and a cusp form. For that reason we briefly recall a standard basis for the Eisenstein subspace of $\mathcal{M}_3(\Gamma_1(N))$, i.e., the orthogonal complement of $\mathcal{S}_3(\Gamma_1(N))$ with respect to the Petersson inner product.

\begin{dfn} A basis for the Eisenstein subspace of $\mathcal{M}_3(\Gamma_1(N))$ is given by
\begin{align*}
E_{3}^{\phi,\psi,t}(\tau)=\delta_{N_1,1} L(\psi,-2) + 2\sum_{m,n\geq 1} \phi(m)\psi(n)n^2 q^{mn}
\end{align*}
for all primitive characters $\phi,\psi$ modulo $N_1,N_2$ such that $\phi(-1)\psi(-1)=-1$ and all $t\in\N$ such that $N_1N_2t$ divides $N$. The completed $L$-function of $E_3^{\phi,\psi,t}$ equals
\begin{align}\label{eqn:eisenstein l-value}
\Lambda(E_3^{\phi,\psi,t},s) = 2t^{-s}N^{s/2}(2\pi)^{-s}\Gamma(s)L(\phi,s)L(\psi,s-2).
\end{align}
\end{dfn}

\subsection{Another example for $\Gamma_1(8)$}
We start with another example for the group $\Gamma_1(8)$. Defining
\[
x(U,V) = \frac{1}{U}, \qquad y(U,V) = \frac{(Z-1)^2U}{ZV} + 1
\]
we get a birational map from the surface defined by \eqref{eqn:weierstrass form E18}, i.e. $V^2+\left(Z+\frac{1}{Z}-2\right) UV = U(U-1)^2$, to the surface $V(Q)$, where
\[
Q(x,y,Z) = (x-1)^2(y-1)^2 - \frac{(Z-1)^4}{Z^2}xy.
\]
Since $V(P)$ and $V(Q)$ are birational, the surface $V(Q)$ is a model for $E_1(8)$. The Deninger cycle $D_Q = \{(x,y,Z)\in V(Q) : |x| = |y| = 1, |Z|>1\}$ associated to $Q$ is very similar to the one we studied in the previous sections. In particular we have the same fibration as before.
\begin{lem}\label{lem:fibration Q} The map $(x,y,Z)\mapsto Z$ defines a fibration of $D_Q\setminus\{(-1,-1,3+2\sqrt{2})\}$ above the interval $(1,3+2\sqrt{2})$. 
\end{lem}
\begin{proof}
Let $(x,y,Z)\in V(Q)$, so
\[
\frac{(x-1)^2}{x}\frac{(y-1)^2}{y} = \left(\frac{(Z-1)^2}{Z}\right)^2.
\]
As in the proof of Lemma \ref{lem:Deninger cycle explicit} we can conclude that for $x,y\in S^1$ and $|Z|>1$ this implies $Z\in (1,3+2\sqrt{2}]$.
\end{proof}
The functions $x$ and $y$ on $E_1(8)$ have horizontal divisors
\begin{equation*}
\dv(x)=2(0)-2(4A), \qquad \dv(y)=2(2A)-2(6A).
\end{equation*}
Note that the divisors of $x$ and $y$ are disjoint. This implies that the horizontal tame symbols of $\{x,y\}$ are trivial and thus $\{x,y\}$ defines an element of $H_{\mathcal{M}}^2(E_1(8),\Q(2))$. Since the elliptic involution maps $(x,y)$ to $(x,1/y)$, the symbol $\{x,y\}$ belongs to $H_{\mathcal{M}}^2(E_1(8),\Q(2))^-$. Furthermore
\[
\pi^*\{x,y\} = \pm(128/3)\cdot\Eis^1(0,2) = 2\pi^*\{X,Y\},
\]
where $\pi$ is the projection from $E(8)$ to $E_1(8)$. Using the fact that $D_Q$, just as $D_P$, is fibred above $(1,3+2\sqrt{2})$, we can follow the method from the last section and write
\[
m(Q) = \pm \frac{1}{(2\pi i)^2}\int_{Y\{1/2,\infty\}} 2\eta
\]
where $\eta$ is defined as in the previous section. Hence we get the following exotic relation between $m(P)$ and $m(Q)$.
\begin{thm} We have
\[
m(Q) = 8\Lambda(f_8,3) = 2m(P).
\]
\end{thm}

\subsection{The surface $E(\Gamma_1(4) \cap \Gamma_0(8)$)}

According to Bertin-Lecacheux \cite[7.1]{Bertin-Lecacheux}, a model of the elliptic modular surface $E(\Gamma)$ associated to the group $\Gamma = \Gamma_1(4) \cap \Gamma_0(8)$ is given by $X+1/X+Y+1/Y=k$ and the elliptic fibration is given by $(X,Y,k) \mapsto k$. The functions $X$ and $Y$ are supported at the torsion sections of $E(\Gamma)$ and the Milnor symbol $\{X,Y\}$ extends to $E(\Gamma)$. The function $k$ is a Hauptmodul for $\Gamma$ and the values of $k$ at the cusps are given by $\{\infty,0,4,-4\}$, so that $k$ is a modular unit. However, the condition (\ref{condition3}) is not satisfied: the image of the Deninger cycle under $k$ is given by $[-4,-1) \cup (1,4]$ and the endpoints $k=\pm 1$ are not cusps. We thus look for another modular unit. Putting $k=4(t+1)$, the cusps are given by $t \in \{\infty,-1,0,-2\}$ so that $t$ is again a modular unit, and the image of the Deninger cycle under $t$ is given by the interval $[-2,-1)$, which joins two cusps. We thus investigate the Mahler measure of the polynomial
\begin{equation*}
R(X,Y,t)=X+\frac{1}{X}+Y+\frac{1}{Y}-4t-4.
\end{equation*}
In terms of the modular unit $Z$ from the $\Gamma_1(8)$  case, we have $k=-Z-1/Z+2$ and thus $4t=-Z-1/Z-2$. Using Proposition \ref{prop:evaluate siegel unit} we find the values of $t$ at the cusps of $\Gamma$, which are represented by $\tau=\infty,0,1/4,1/2$: $t(\infty)=-2,t(0)=\infty,t(1/4)=0,t(1/2)=-1$. The divisor of $t$ is thus $(1/4)-(0)$ from which we find
\begin{equation*}
4t = i \cdot \frac{g_{0,2}^6 g_{0,4}^2}{g_{0,1}^4 g_{0,3}^4} = 24 \Eis^0(0,2) +8\Eis^0(0,4) - 16 \Eis^0(0,1) - 16\Eis^0(0,3).
\end{equation*}
Letting $\pi$ be the canonical map $E_1(8) \to E(\Gamma)$, we have as in Proposition \ref{pro piXY}
\begin{equation*}
\pi^* \{X,Y\} = \pm \frac{64}{3} \cdot \Eis^1(0,2).
\end{equation*}
Jensen's formula gives
\begin{equation*}
m(R)-\log 4 = \frac{1}{(2\pi i)^2} \int_{D_R} \log |t| \frac{dX}{X} \wedge\frac{dY}{Y}.
\end{equation*}
The map $(X,Y,t) \mapsto (X,Y)$ identifies $D_R$ with the subset of the torus $T^2$ defined by the condition $\Re(X+Y)<0$, hence
\begin{equation*}
\int_{D_R} \frac{dX}{X} \wedge\frac{dY}{Y} = \frac12 \int_{T^2}  \frac{dX}{X} \wedge\frac{dY}{Y} = \frac12 (2\pi i)^2.
\end{equation*}
It follows that
\begin{equation*}
m(R)-\log 4 = \frac{1}{(2\pi i)^2} \int_{D_R} \log |4t| \frac{dX}{X} \wedge\frac{dY}{Y} - \frac12 \log 4.
\end{equation*}
As in Section \ref{section:Calculation}, the Deninger cycle $D_R$ is homologous to $Y \{1/2,\infty\}$ and similar computations give
\begin{equation*}
m(R)=\log 2 + \Lambda^*(4 E_3^{\chi_4,\textbf{1},1}-32 E_3^{\chi_4,\textbf{1},2},0).
\end{equation*}
where $\chi_4$ is the nontrivial Dirichlet character of conductor 4. Using (\ref{eqn:eisenstein l-value}), we get
\begin{equation*}
\Lambda(E_3^{\chi_4,\textbf{1},1},0) = \Lambda(E_3^{\chi_4,\textbf{1},2},0) = 2 L(\chi_4,0) \zeta'(-2) = -\frac{\zeta(3)}{4\pi^2}.
\end{equation*}
This gives the following theorem.

\begin{thm} We have
\[
m(R) = \log 2 + \frac{7\zeta(3)}{\pi^2}.
\]
\end{thm}

\subsection{The surface $E_1(6)$}
In contrast to $E_1(8)$, the surface $E_1(6)$ is not a K3 surface but a rational surface. This is equivalent to the fact that $\mathcal{S}_3(\Gamma_1(6))=\{0\}$. We will show that the Mahler measure of a suitable model of $E_1(6)$ is the $L$-value of an Eisenstein series for $\Gamma_1(6)$.

An affine model for the universal elliptic curve for $\Gamma_1(6)$ can be found in \cite[Table $1$]{verrill} and \cite[\S 4]{BlochVanhove}:
\[
P_6(x,y,s) = s - (x+y+1)\left(\frac{1}{x}+\frac{1}{y}+1\right) = 0,
\]
and the universal $6$-torsion point is given by $A=(-1,0)$.
The birational transformation
\begin{align*}
(x,y) &= \left(\frac{2Y+X(s-1)}{2X(1-X)}, \frac{-2Y+X(s-1)}{2X(1-X)}\right),\\
(X,Y) &= \left(\frac{x+y-s+1}{x+y},\frac{(s-1)(x-y)(x+y-s+1}{2(x+y)^2}\right)
\end{align*}
puts $V(P_6)$ into Weierstrass form
\[
E_s: Y^2 = X^3 + \frac{(s-3)^2-12}{4}X^2 +sX
\] 
and $(X(A),Y(A)) = (s, \frac12(s^2 - s))$. The functions $x$ and $y$ have horizontal divisors
\begin{equation}\label{eqn:E16-horizontal divs}
\dv(x) = (0) -(2A) -(3A)+(5A), \qquad \dv(y)= (0) + (A) - (3A)-(4A).
\end{equation}
We can follow the method in Section \ref{section:elliptic modular surface} to find a birational map from $E_1(6)$ to $V(P_6)$:
\begin{align*}
X &= \frac{\wp_\tau(z) - \wp_\tau(1/2)}{\wp_\tau(1/3)-\wp_\tau(1/2)},\\
Y &= \frac{(\wp_\tau(z)-\wp_\tau(1/3))^2 (\wp_\tau(z+1/3)-\wp_\tau(z-1/3))}{2(\wp_\tau(1/3)-\wp_\tau(1/6))(\wp_\tau(1/3)-\wp_\tau(1/2))^2},\\
s &= \frac{\wp_\tau(1/6) - \wp_\tau(1/2)}{\wp_\tau(1/3)-\wp_\tau(1/2)}=-\zeta_6\frac{g_{0,2}(\tau)^4}{g_{0,1}(\tau)^4}.
\end{align*}
The function $s$ is a Hauptmodul for $\Gamma_1(6)$. From its expression as a quotient of Siegel units we get
\[
\pi^* s = 12\text{Eis}^0(-(0,1) + (0,2)).
\]
where $\pi$ denotes the projection map from $E(6)$ to $E_1(6)$. The values of $s$ at the cusps of $\Gamma_1(6)$ are $s(0)=\infty, s(1/2)=1, s(1/3)=0$, and $s(\infty)=9$.

The elliptic involution on $V(P_6)$ maps $(X,Y)$ to $(X,-Y)$ and hence $(x,y)$ to $(y,x)$. Hence we can proceed as in Proposition 6.10 and deduce
\[
\pi^*\{x,y\} = \pm 12\text{Eis}^1((0,1) + (0,2))
\]
\begin{lem}
The fibration $(x,y,s)\mapsto s$ restricts to a fibration of Deninger cycle $\{(x,y,s)\in V(P_6) : |x|=|y|=1,~|s|>1\}\setminus \{(1,1,9)\}$ above $(1,9)$.
\end{lem}
\begin{proof}
For $(x,y,s)\in V(P_6)$ the assumption $x,y\in S^1$ implies $s = 1+2\Re((x+1)(\overline{y}+1))$, so $s\in (1,9]$ and $s=9$ is only attained at the excluded point $(1,1,9)$.
\end{proof}
The points $s=1$ and $s=9$ correspond to $\tau= \frac12$ and $\tau=\infty$. Hence as before we can conclude that
\[
m(P_6) = \pm\frac{1}{(2\pi i)^2}\int_{Y\{1/2,\infty\}} 12^2 \text{Eis}^{0,1}\big(((0,1) - (0,2))((0,1) + (0,2))\big).
\]
The main theorem of \cite{brunault:regRZ} implies $m(P_6)= \frac{1}{72} \Lambda^*(-432 E_3^{\chi_3,\textbf{1}} + 3456 E_3^{\chi_3,\textbf{1},2}, 0)$ where $\chi_3$ is the nontrivial Dirichlet character of conductor 3. Using \eqref{eqn:eisenstein l-value}:

\begin{thm} We have
\[
m(P_6) = \frac{7\zeta(3)}{\pi^2}.
\]
\end{thm}

\begin{rem}
The value $7\zeta(3)/(2\pi^2)$ appeared first as the Mahler measure of $1+X+Y+Z$, as shown by Smyth, see \cite{Boyd1981}. It also appears in several Mahler measures of elliptic surfaces as above, see \cite{Lalin2008}. For another example
consider the Hesse pencil, the universal elliptic curve for $\Gamma(3)$. By a simple change of variables we get $m(X^3+Y^3+1+ZXY) = m(1+X+Y+Z) = 7\zeta(3)/(2\pi^2)$.
\end{rem}
\subsection{The surface $E(2,6)$}

From the model of $E_1(6)$ we can construct a model for $E(2,6)$, the universal elliptic curve for the group $\Gamma(2,6) = \Gamma_1(6)\cap\Gamma(2)$. We look for a base change of $E_1(6)$ of the form $s=f(t)$ such that $E_{f(t)}$ acquires full $2$-torsion over $\Q(t)$. This condition is equivalent to the polynomial $X^2+\frac{(s-3)^2-12}{4}X+s$ being split over $\Q(t)$, which amounts to say that
\[
\frac{s^4-12s^3+30s^2-28s+9}{16} = \frac{(s-9)(s-1)^3}{16}
\]
is a square in $\Q(t)$. We find that the base change $s = f(t) = 2(t+1/t)+5$ does the job.

We will determine the Mahler measure of
\[
P_{2,6}(x,y,t) = P_6(x,y, 2(t+1/t)+5)=2\left(t+\frac{1}{t}\right)+5 - (x+y+1)\left(\frac{1}{x}+\frac{1}{y}+1\right) = 0.
\]
By construction $E'_t = E_{f(t)}$ is a Weierstrass equation for $E_{2,6}$. The torsion subgroup of $E'_t$ is generated by the points $A_1,A_2$ of orders $2,6$ respectively, given by
\begin{align*}
A_1 & = \left(\frac{-2t-1}{t^2},0\right)\\
A_2 & = \left(\frac{2t^2 + 5t + 2}{t},\frac{2t^4 + 9t^3 + 14t^2 + 9t + 2}{t^2}\right).
\end{align*}
Note that $A_2$ is the pull-back of $A$ under the base change map $E(2,6) \to E_1(6)$. We choose a modular parametrisation that is compatible with the one we chose for $E_1(6)$, so that $A_1$ corresponds to $z=\tau/2$ and $A_2$ corresponds to $z=1/6$. We find
\[
t^2 = \frac{X(A_1+3A_2)-X(2A_2)}{X(A_1)-X(2A_2)}\\
= \frac{\wp_\tau(\tau/2+1/2)-\wp_\tau(1/3)}{\wp_\tau(\tau/2)-\wp_\tau(1/3)}=\frac{g_{3,1}^2 g_{3,0}^2}{g_{3,3}^2 g_{3,2}^2}
\]
so that in $\mathcal{O}(Y(2,6))^\times \otimes \Q$ we have
\begin{equation*}
t = \frac{g_{3,0} g_{3,1}}{g_{3,2} g_{3,3}} = 3\Eis^0((3,0)+(3,1)-(3,2)-(3,3)).
\end{equation*}
Notice that knowing $t^2$ and $s$ is enough to determine $t = 2 (t^2+1)/(s-5)$. In particular we may compute numerically $t(\tau)$ without any sign ambiguity. A set of representatives of the cusps of $\Gamma_1(6) \cap \Gamma(2)$ is given by $\infty$, $0$, $1/5$, $1/4$, $1/3$, $2/3$. Under $t$ these are mapped respectively to $1$, $0$, $\infty$, $-1$, $-1/2$, $-2$. Also we have $t(1/2)=-1$.

\begin{lem}
Let $D_{2,6}$ be the Deninger cycle associated to $P_{2,6}$. The map $(x,y,t) \mapsto t$ endows $D_{2,6}$ with a fibration over the interval $[-2,-1)$.
\end{lem}

\begin{proof}
From the equation $P_{2,6}(x,y,t)=0$ we deduce $f(t) \in [0,9]$ and thus $t+1/t \in [-5/2,2]$. Together with $|t|>1$ this implies that $t$ is real and $t \in [-2,-1)$.
\end{proof}

This shows that $D_{2,6}$ is fibred over the modular symbol $\{1/2,2/3\}$. Proceeding an in the previous cases, we find that $D_{2,6}$ is homologous to $Y\{1/2,2/3\}$.

The horizontal divisors of $x,y$ viewed as rational functions on $E'_t$ can be obtained from \eqref{eqn:E16-horizontal divs} by replacing $(kA)$ with its pullback $(kA_2)$:
\[
\dv(x)  = (0)-(2A_2)-(3A_2)+(5A_2), \qquad
\dv(y)  = (0)+(A_2)-(3A_2)-(4A_2).
\]
As in the case of $E_1(6)$, the Eisenstein symbol on $E(6)$ corresponding to $\pi^*\{x,y\}$ is
\begin{equation*}
\pi^* \{x,y\} = \pm 12 \Eis^1((0,1)+(0,2)).
\end{equation*}
Since $P_{2,6}^*=2$, the difference $m(P)-\log 2$ is the integral over $Y\{1/2,2/3\}$ of the symbol
\begin{equation*}
\eta = \pm 36 \Eis^{0,1}((3,0)+(3,1)-(3,2)-(3,3),(0,1)+(0,2)).
\end{equation*}
Notice that $\{1/2,2/3\}=\{g0,g\infty\}$ with $g=\begin{pmatrix} 2 & 1 \\ 3 & 2 \end{pmatrix} \in \SL_2(\Z)$. So this is also the integral over $X\{0,\infty\}$ of
\begin{equation*}
\eta' = 3 g^* \eta - 2 (g\sigma)^* \eta.
\end{equation*}
At this point we conclude as before with the main theorem of \cite{brunault:regRZ},
\[
m(P)-\log(2)= - \Lambda\left(\frac14E-\frac32f_{12},0\right),
\]
where $f_{12} = q - 3q^3 + 2q^7 + O(q^9)$ is the unique newform in $\mathcal{S}_3(\Gamma_1(12))$ with rational Fourier coefficients and $E=E_3^{\textbf{1},\chi_3} + 7E_3^{\textbf{1},\chi_3,2} - 8E_3^{\textbf{1},\chi_3,4}$. It remains to find $\Lambda(E,0)$ which, by \eqref{eqn:eisenstein l-value}, equals
\[
2L(\chi_3,-2)\zeta(0)\lim_{s\to 0}\Gamma(s)(1+7\cdot 2^{-s} - 8\cdot 4^{-s}) = 2\log(2).
\]
In summary:
\begin{thm} We have
\[
m(P_{2,6}) = \frac32\Lambda(f_{12},0) + \frac{\log(2)}{2}.
\]
\end{thm}

\subsection{The surface $E(4)$}

For the universal elliptic curve $E(4)$ attached to the group $\Gamma(4)$, we proceed slightly differently from the $\Gamma_1(8)$ case: we avoid the use of Stokes' theorem and the consideration of residues at the cusps, using instead a distribution relation satisfied by Eisenstein symbols.

According to the Rouse--Zureick-Brown tables \cite{rouse--zureick-brown}, a model of $E(4)$ over $\Q(i)$ is given by
\[
E(4) : y^2=x^3+Ax+B
\]
with $A=-27t^8-378t^4-27$ and $B=54t^{12}-1782t^8-1782t^4+54$. Here $t$ is a generator of the function field of $Y(4)$, more precisely we have $\Q(Y(4))=\Q(i)(t)$. Since the singular fibres of $E(4)$ are $t=0, \pm 1, \pm i, \infty$, it follows that $t$ is a modular unit. In the following Lemma, we express $t$ as a quotient of Siegel units.

\begin{lem}\label{lem t}
We have $t=i g_{1,2}^2/g_{1,0}^2$.
\end{lem}

\begin{proof}
The torsion subgroup of $E(4)$ over $\Q(i)(t)$ is isomorphic to $\Z/4\Z \times \Z/4\Z$, generated by the two points of order 4
\begin{align*}
P_1 & = (3t^4 + 18it^3 - 18t^2 - 18it + 3, -54it^5 + 108t^4 + 108it^3 - 108t^2 - 54it),\\
P_2 & = (-15t^4 + 3, -54it^6 + 54it^2).
\end{align*}
From these expressions we find that
\begin{equation*}
t = \frac{x(P_1+P_2)-x(2P_1+2P_2)}{x(2P_1+2P_2)-x(2P_1+P_2)}.
\end{equation*}
In particular $t$ is a Weierstrass unit and we may express it as a quotient of Siegel units. We parametrize $E(4)$ in such a way that $P_1$, $P_2$ correspond to $z_1=1/4$, $z_2= \tau/4$ respectively (this identification is compatible with the Weil pairing). This gives
\begin{equation*}
t(\tau) = \frac{\wp_\tau(\tau/4+1/4)-\wp_\tau(\tau/2+1/2)}{\wp_\tau(\tau/2+1/2)-\wp_\tau(\tau/4+1/2)} = i \frac{g_{3,3}g_{1,2}^2}{g_{1,1}g_{1,0}g_{3,0}}=i\frac{g_{1,2}^2}{g_{3,0}^2}.
\end{equation*}
Proceeding as in the case of $E_1(8)$, we deduce the Lemma.
\end{proof}

Since we know how to evaluate Siegel units at any cusp, we get the following result.

\begin{lem}
We have $t(\infty)=i$ and $t(1/2)=t(-1/2)=-i$.
\end{lem}

Using Magma, we find two functions $X,Y$ on $E(4)$ such that $\{X,Y\} \in H^2_\mathcal{M}(E(4),\Q(2))$. Their divisors are given by
\begin{equation*}
\dv(X) = -2(2P_2) + 2(2P_1+2P_2) \qquad \dv(Y) = -2(P_1) + 2(3P_1)
\end{equation*}
and they satify the relation
\begin{equation*}
P(X,Y,t) =  (X+\frac{1}{X}-2)(Y+\frac{1}{Y}+2) + 2i \frac{(t-i)^4}{t^3 - t} = 0.
\end{equation*}
Let $D \subset V(P)$ be the Deninger cycle associated to $P$ with respect to the variable $t$. By Jensen's formula, we have
\begin{equation}\label{mP E4}
m(P)-m(P^*) = \frac{1}{(2\pi i)^2} \int_D \log |t| \cdot \frac{dX}{X} \wedge \frac{dY}{Y} = \frac{1}{(2\pi i)^2} \int_D \log |t| \cdot \pi_2(\frac{dX}{X} \wedge \frac{dY}{Y}).
\end{equation}
We now want to relate $D$ with a Shokurov cycle on the connected component of $E(4)(\C)$ defined as the image of the map $(\tau,z) \mapsto (\tau,z,\sigma)$.

\begin{lem}\label{two arcs}
Let $C_1$ (resp. $C_2$) be the circle with center $1$ (resp. $-1$) passing through $i$ in the complex plane. Then $t(D)$ is the union of the two arcs $C_1 \cap \{|t|>1\}$ and $C_2 \cap \{ |t|>1\}$.
\end{lem}

\begin{proof}
Let $(X,Y,t) \in D$. Since $|X|=|Y|=1$, we have $2i(t-i)^4/(t^3-t) \in \R$. Writing $t=u+iv$, we find the equation
\begin{equation*}
2u(u^2+v^2+2u-1)(u^2+v^2-2u-1)(u^2+v^2-1)=0.
\end{equation*}
We have $X+1/X-2 \in [-4,0]$ and $Y+1/Y+2 \in [0,4]$, thus $2i(t-i)^4/(t^3-t) \in [0,16]$. The equation $u=0$ leads to $-2(v-1)^4/(v^3+v) \in [0,16]$ which is impossible for $|v|>1$. Therefore $t(D)$ is contained in the union of the two arcs.

Conversely, let $t \in C_1 \cup C_2$ with $|t|>1$. Since $D$ is invariant under $(X,Y,t) \mapsto (\overline{X},\overline{Y},-\overline{t})$, it suffices to treat the case $t \in C_1$. So let us write $t=1+\sqrt{2} e^{i\theta}$ with $\theta \in (-3\pi/4,3\pi/4)$. We find
\begin{align*}
2i \frac{(t-i)^4}{t^3-t} & = \frac{2i(t-i)^4(\overline{t}^3-\overline{t})}{|t^3-t|^2} \\
& = \frac{416+384\sqrt{2} \cos \theta-224\sqrt{2}\sin \theta + 128 \cos(2\theta)-224 \sin(2\theta)-32\sqrt{2} \sin(3\theta)}{52+48\sqrt{2} \cos \theta + 16 \cos(2\theta)}.
\end{align*}
An explicit computation reveals that this function takes values in $[0,16]$, hence there exist $X,Y \in S^1$ such that $P(X,Y,t)=0$.
\end{proof}

\begin{lem}
The map $\tau \mapsto t(\tau)$ is a diffeomorphism from the modular symbols $\{-1/2,\infty\}$ and $\{1/2,\infty\}$ to $C_1 \cap \{|t|>1\}$ and $C_2 \cap \{|t|>1\}$ respectively.
\end{lem}

\begin{proof}
We want to show that for $\tau \in \{1/2,\infty\}$, the point $t(\tau)$ lies on $C_2$, that is $|t(\tau)+1|^2=2$.
Using $t(-\overline{\tau}) = -\overline{t(\tau)}$ we can then conclude that $t(\{-1/2,\infty\})$ lies on $C_1$.

The functions $t+1$ and $t-1$ are modular units, since $t(1)=-1$ and $t(3)=1$. Their expression in terms of Siegel units is given by
\[
t+1 = g_{0,3}g_{1,0}^{-2}g_{1,1}^{2}g_{1,3}^{4}g_{2,1},\quad t-1 = -2 g_{0,3}^{-1}g_{1,0}^{-4}g_{1,1}^2g_{1,2}^{-2}g_{2,3}^{-1}.
\]
From Proposition \ref{prop:siegel-transformation} we see that $\sabcd 1{-1}01^*(t-1) = -2(t+1)^{-1}$. So for $v\in\R$ \[
\overline{t(1/2+iv)+1}=-(t(-1/2+iv)-1)=-\sabcd 1{-1}01^*(t-1)(1/2+iv)=\frac{2}{t(1/2+iv)+1},
\]
which shows $|t(1/2+iv)+1|^2=2$.

It remains to prove that $|t(1/2+iv)|>1$. For $\tau=1/2+iv$ we have
\begin{equation*}
|t(\tau)|^2 = t(\tau) \overline{t(\tau)} = -t(\tau)t(-\bar{\tau}) = -t(\tau)t(\tau-1)= \frac{t(\tau) (1-t(\tau))}{1+t(\tau)}.
\end{equation*}
So $|t(\tau)|=1$ implies $t(\tau)-t(\tau)^2=1+t(\tau)$ so that $t(\tau)=\pm i$. But these values correspond to cusps, so we get a contradiction. To conclude, it suffices to show that $|t(1/2+iv)|>1$ for one particular value of $v$, which we can check numerically.

To show that the map $t:\{\frac12,\infty\}\to C_2\cap\{|t|>1\}$ is a diffeomorphism, it remains to show that it is injective. Since $t$ is a Hauptmodul for $\Gamma(4)\backslash\h$, it suffices to show that the projection from $\{\frac12,\infty\}$ to $\Gamma(4)\backslash\h$ is injective. In fact even $\{\frac12,\infty\} \to \Gamma_1(4)\backslash\h$ is injective which can be shown with \cite[Lemma 3.3.4]{snowden:real}, as in Lemma \ref{diffeo Ztau}.
\end{proof}

By the previous lemmas, we see that $D = D_1 \cup D_2$ where $D_1$ (resp. $D_2$) is fibred over $\{-1/2,\infty\}$ (resp. $\{1/2,\infty\}$). Moreover, a numerical computation reveals that $D_2$ is fibrewise homologous to the Shokurov cycle $\gamma = (2X-Y)\{1/2,\infty\}$. Let us consider the complex conjugation $c_0 : (\tau,z,\sigma) \mapsto (-\overline{\tau},\overline{z},\sigma)$ on $E(4)(\C)$. This involution corresponds to the map $(X,Y,t) \mapsto (\overline{X},1/\overline{Y},-\overline{t})$ on $V(P)$. It thus exchanges $D_1$ and $D_2$ and reverses the orientations. It follows that $D$ is fibrewise homologous to $\gamma - c_0(\gamma)$ as an oriented cycle.

Let us now turn to the integrand of (\ref{mP E4}). By similar arguments as in Lemma \ref{mP eta1}, we have
\begin{equation*}
m(P)-m(P^*) = \pm \frac{1}{(2\pi i)^2} \frac{256}{3} \int_D \eta
\end{equation*}
with
\begin{equation*}
\eta = (\Eis^0_{\mathcal{D}}(1,2)-\Eis^0_{\mathcal{D}}(1,0)) \cdot \pi_2(\Eis^1_{\mathrm{hol}}(2,1)).
\end{equation*}
Note that the differential form $\operatorname{darg}{t}$ does not vanish on $D$, so we cannot write the Mahler measure as an integral of a closed differential form as in Lemma \ref{mP eta1}. Since $\eta$ is holomorphic with respect to $z$, and since $c_0^* \eta = -\eta$, we have
\begin{equation*}
\int_D \eta = \int_{\gamma} \eta - \int_{c_0(\gamma)} \eta = 2 \int_\gamma \eta.
\end{equation*}
Note that $\gamma=(2X-Y) \{1/2,\infty\}$ is a sum of two Manin symbols, but since $\eta$ is not closed, we cannot apply Stokes' theorem as in \S \ref{section:Calculation}. Instead we use the following degeneracy map from $E(8)$ to $E(4)$:
\begin{align*}
\lambda : (\Z^2 \rtimes \Gamma(8)) \backslash (\mathcal{H} \times \C) & \to (\Z^2 \rtimes \Gamma(4)) \backslash (\mathcal{H} \times \C) \\
(\tau,z) & \mapsto \left(\frac{\tau+1}{2},z\right).
\end{align*}
We have $\gamma = \lambda_* X\{0,\infty\}$ so that
\begin{equation*}
\int_{\gamma} \eta = \int_{\lambda_* X\{0,\infty\}} \eta = \int_{X\{0,\infty\}} \lambda^* \eta.
\end{equation*}
Using the Fourier expansion of Eisenstein symbols \cite[\S 8]{brunault:regRZ}, we compute
\begin{align*}
\lambda^* \Eis^{0,4}_{\mathcal{D}}(a,b) & = 2 \sum_{\substack{a' \in \Z/8\Z \\ a' \equiv a \pmod{4}}} \Eis^{0,8}_{\mathcal{D}}(a',a'+2b), \\
\lambda^* \Eis^{1,4}_{\mathrm{hol}}(a,b) & = 4 \sum_{\substack{a' \in \Z/8\Z \\ a' \equiv a \pmod{4}}} \Eis^{1,8}_{\mathrm{hol}}(a',a'+2b),
\end{align*}
where $\Eis^{k,N}$ denotes the Eisenstein symbol of weight $k$ and level $N$. It follows that
\begin{equation*}
\lambda^* \eta = 8 (\Eis^{0,8}_{\mathcal{D}}(1,5)+\Eis^{0,8}_{\mathcal{D}}(5,1)-\Eis^{0,8}_{\mathcal{D}}(1,1)-\Eis^{0,8}_{\mathcal{D}}(5,5)) \cdot \pi_2(\Eis^{1,8}_{\mathrm{hol}}(2,4)+\Eis^{1,8}_{\mathrm{hol}}(6,0)).
\end{equation*}
By Proposition \ref{prop:abs convergence}(2) the integral of $\lambda^* \eta$ is convergent. Applying the formulas \cite[top of p. 1149]{brunault:regRZ} and \cite[(36)]{brunault:regRZ}, we obtain $m(P)-m(P^*) = \pm \Lambda(F(2\tau),0)$ with
\begin{equation*}
F = 4f_{16} - 2E_3^{1,\chi_4}+2E_3^{1,\chi_4,2}.
\end{equation*}
Since $m(P^*)=\log 2 = \Lambda(2E_3^{1,\chi_4}-2E_3^{1,\chi_4,2},0)$, we finally get $m(P)=4\Lambda(f_{16},0)$. The change of variables $(X,Y,t) \to (-X,Y,-it)$ gives the last identity of Theorem \ref{thm:Summary}.

%
%Bibliography
%

\bibliographystyle{plain}
%\bibliography{papers}
\bibliography{refs}

\vspace{\baselineskip}

\author{Fran\c{c}ois Brunault, francois.brunault@ens-lyon.fr\\
\'{E}NS Lyon, UMPA, 46 all\'{e}e d'Italie, 69007 Lyon, France \\
\\
Michael Neururer, neururer@mathematik.tu-darmstadt.de\\
TU Darmstadt, Schlo\ss gartenstr. 7, 64289 Darmstadt, Germany}

\end{document}